\documentclass{birkjour}

\usepackage[T1]{fontenc}
\usepackage[utf8]{inputenc}
\usepackage{amsmath,amsfonts,amsthm,amssymb,amsxtra,bbm}
\usepackage{fourier,setspace,graphicx,color,pdflscape}
\usepackage[colorlinks=true]{hyperref}
\hypersetup{urlcolor=blue, linkcolor=blue, citecolor=red, anchorcolor=blue]}
\newtheorem{thm}{Theorem}[section]

\newtheorem{lem}[thm]{Lemma}
\newtheorem{prop}[thm]{Proposition}
\theoremstyle{definition}

\theoremstyle{remark}

\numberwithin{equation}{section}

\newcommand{\R}{{\mathbb R}}
\newcommand{\N}{{\mathbb N}}

\renewcommand{\S}{{\mathbb S}}
\renewcommand{\P}{{\mathsf P}}
\newcommand{\be}[1]{\begin{equation}\label{#1}}
\newcommand{\ee}{\end{equation}}
\renewcommand{\(}{\left(}
\renewcommand{\)}{\right)}
\newcommand{\ird}[1]{\int_{\R^d}{#1}\,dx}
\newcommand{\nrm}[2]{\|{#1}\|_{#2}}

\newcommand{\nrmSd}[2]{\|{#1}\|_{\mathrm L^{#2}(\mathbb S^d)}}
\newcommand{\iSd}[1]{\int_{\mathbb S^d}{#1}\,d\mu}
\newcommand{\irdmu}[1]{\int_{\R^d}{#1}\,d\mu}

\newcommand{\D}[1]{\mathsf D_\alpha\kern1pt#1}
\newcommand{\Dstar}[1]{\mathsf D_\alpha^*\kern1pt#1}
\newcommand{\mB}{\mathcal B}
\newcommand{\Mstar}{\mathcal M}
\newcommand{\taustar}{{\mathsf c_\star}}

\newcommand{\nrmS}[2]{\|{#1}\|_{\mathrm L^{#2}(\S^d)}}

\newcommand{\conj}[1]{\\[2pt]\centerline{\emph{#1}}\\[2pt]}
\newcommand{\msc}[1]{\href{https://mathscinet.ams.org/mathscinet/search/mscbrowse.html?sk=default&sk=#1&submit=Chercher}{#1}}

\definecolor{darkred}{rgb}{0.8,0,0}

\begin{document}

\title[Nonlinear flows and entropy methods in functional inequalities]
{Functional inequalities: nonlinear flows and entropy methods as a tool for obtaining sharp and constructive results}

\author[J.~Dolbeault]{Jean Dolbeault}
\address{CEREMADE (CNRS UMR n$^\circ$ 7534), PSL university, Universit\'e Paris-Dauphine, Place de Lattre de Tassigny, 75775 Paris~16, France}
\email{dolbeaul@ceremade.dauphine.fr}


\date{\today}

\begin{abstract} Interpolation inequalities play an essential role in Analysis with fundamental consequences in Mathematical Physics, Nonlinear Partial Differential Equations (PDEs), Markov Processes, \emph{etc.}, and have a wide range of applications in various other areas of Science. Research interests have evolved over the years: while mathematicians were originally focussed on abstract properties (for instance appropriate notions of functional spaces for the existence of weak solutions in PDEs), more qualitative questions (for instance, bifurcation diagrams, multiplicity of the solutions in PDEs and their qualitative behaviour) progressively emerged. The use of entropy methods in nonlinear PDEs is a typical example: in some cases, the optimal constant in the inequality can be interpreted as an optimal rate of decay of an entropy for an associated evolution equation. Much more has been learned by adopting this point of view.

This paper aims at illustrating some of these recent aspect of entropy-entropy production inequalities, with applications to stability in Gagliardo-Nirenberg-Sobolev inequalities and symmetry results in Caffarelli-Kohn-Nirenberg inequalities. Entropy methods provide a framework which relates nonlinear regimes with their linearized counterparts. This framework allows to prove optimality results, symmetry results and stability estimates. Some emphasis will be put on the hidden structure which explain such properties. Related open problems will be listed.\end{abstract}

\maketitle

%%%%%%%%%%%%%%%%%%%%%%%%%%%%%%%%%%%%%%%%%%%%%%%%%%%%%%%%%%%%%%%%%%%%%%%%
%%%%%%%%%%%%%%%%%%%%%%%%%%%%%%%%%%%%%%%%%%%%%%%%%%%%%%%%%%%%%%%%%%%%%%%%
\newpage
The two main results of Sections~\ref{Sec:FDE} and~\ref{Sec:CKN} are based on recent papers. Stability results on Gagliardo-Nirenberg-Sobolev inequalities have been obtained in collaboration with M.~Bonforte, B.~Nazaret, and N.~Simonov and are collected in~\cite{BDNS2021}. Symmetry results in Caffarelli-Kohn-Nirenberg inequalities rely on a series of papers with M.J.~Esteban, M.~Loss and M.~Muratori, see \cite{Dolbeault2016,Bonforte2017a,Dolbeault2017}. The guideline of this paper is to clarify the interplay of nonlinear interpolation inequalities with their linearized counterparts using entropy methods and fast diffusion flows.

%%%%%%%%%%%%%%%%%%%%%%%%%%%%%%%%%%%%%%%%%%%%%%%%%%%%%%%%%%%%%%%%%%%%%%%%
%%%%%%%%%%%%%%%%%%%%%%%%%%%%%%%%%%%%%%%%%%%%%%%%%%%%%%%%%%%%%%%%%%%%%%%%
\section{A brief historical perspective}\label{Sec:History}

The goal of this section is to sketch important steps in the development of the theory of some important functional inequalities, in the perspective of entropy methods. It is neither a complete review nor a detailed historical account, but more a personal selection of some results which are highlighted as representative of the evolution of the ideas. Several results have been discovered at least twice, in completely independent papers. Whenever possible, this has been clarified, but such issues are always delicate and the overall information is probably still incomplete.

%%%%%%%%%%%%%%%%%%%%%%%%%%%%%%%%%%%%%%%%%%%%%%%%%%%%%%%%%%%%%%%%%%%%%%%%
\subsection{Sobolev and some other interpolation inequalities}

Unless it is specified, we consider functions on $\R^d$. For any $q\in[1,\infty)$, $\nrm fq:=\(\ird{|f|^q}\)^{1/q}$, whilst $\nrm f\infty$ denotes the $\mathrm L^\infty(\R^d)$ norm. 

%.......................................................................
\subsubsection{Sobolev inequality}

On $\R^d$ with $d\ge3$, the classical \emph{Sobolev inequality} asserts that
\be{SobolevRd}
\nrm{\nabla f}2^2\ge\mathsf S_d\,\nrm f{2^*}^2
\ee
for any smooth and compactly supported function, where $2^*=2\,d/(d-2)$ is the critical exponent and the optimal constant is
\[
\mathsf S_d=\pi\,d\,(d-2)\,\Big(\tfrac{\Gamma\(d/2\)}{\Gamma\(d\)}\Big)^{2/d}=\frac14\,d\,(d-2)\,\frac{2^\frac2d\,\pi^{1+\frac1d}}{\Gamma\big((d+1)/2\big)^{2/d}}=\frac14\,d\,(d-2)\,\big|\mathbb S^d\big|^\frac2d\,.
\]
Inequality~\eqref{SobolevRd} is easily extended to the Beppo-Levi space using standard completion results with respect to the norm $f\mapsto\(\nrm{\nabla f}2^2+\nrm f{2^*}^2\)^{1/2}$, see for instance~\cite{AIF_1954_5_305_0}. Inequality~\eqref{SobolevRd} appears in~\cite{soboleff1938theoreme} and is extremely well understood: see~\cite{Aubin-76,MR601601,MR699419}. The question of the optimality of the constant~$\mathsf S_d$ has anyway a much longer history.

Among radial functions and at least in dimension $d=3$ for the critical exponent $2^*=6$, the solutions of the Euler-Lagrange equation enter in the study of the polytropic and isothermal gas spheres, which are models of gaseous stars in astrophysics and can be traced back to the work of H.~Lane in~\cite{Lane_1870}. A detailed account can be found in~\cite[Chapter IV]{MR0092663}. The explicit computation of the optimal radial functions and the corresponding value of $\mathsf S_d$ appears in~\cite{bliss1930integral}. Further considerations on the linear stability of the radial optimal functions can be found in~\cite{MR0289739}.

The classical papers for the optimal Sobolev inequality are those of T.~Aubin and G.~Talenti~\cite{Aubin-76,Talenti-76}, with an earlier contribution by E.~Rodemich in~\cite{rodemich1966sobolev} which is quoted for instance in~\cite[p.~158]{MR1814364}. These papers rely on \emph{spherically symmetric rearrangements} and particularly on the P\'olya–Szeg\H o principle: we refer to~\cite{MR810619,MR2238193,MR1322313} and references therein for details. An alternative proof based on Riesz' lemma and duality can be found in~\cite{MR471785,MR1817225}. Many other proofs have been established using various techniques like, \emph{e.g.}, mass transport in~\cite{MR2032031}. Some results based on nonlinear flows and \emph{entropy methods}, which completely avoid spherically symmetric rearrangements, will be exposed in Section~\ref{Sec:FDE}.

Optimal functions in~\eqref{SobolevRd} are known as the \emph{Aubin-Talenti functions}. The set~$\mathfrak M$ of these functions is a $(d+2)$-dimensional manifold generated by dilations, translations and multiplications by a constant of
\[
\mathsf g(x):=\(1+|x|^2\)^{-(d-2)/2}\quad\forall\,x\in\R^d\,.
\]
Stability is a delicate issue. As an answer to some of the questions raised in~\cite{MR790771}, G.~Bian\-chi and H.~Egnell proved in~\cite{MR1124290} the existence of~$\mathcal C>0$ such that
\be{Bianchi-Egnell}
\nrm{\nabla f}2^2-\mathsf S_d\,\nrm f{2^*}^2\ge\mathcal C\,\inf_{g\in\mathfrak M}\nrm{\nabla f-\nabla g}2^2\,,
\ee
The existence of $\mathcal C$ is obtained by concentration-compactness methods and by contradiction, hence comes with no estimate. Various refinements starting with~\cite{Cianchi2009} and the more recent results of~\cite{Figalli_2018,MR4048334,figalli2020sharp} have been obtained, although without explicit and constructive estimates on $\mathcal C$. By duality methods inspired by~\cite{MR717827} and flows, quantitative and constructive results were obtained in weak norms in~\cite{MR2915466,MR3227280}. In~\cite{BDNS2021}, a result of stability has been obtained, with a measure of the distance to $\mathfrak M$ by a relative Fisher information, a strong norm: this result will be presented in Section~\ref{Sec:FDE}.

%.......................................................................
\subsubsection{Gagliardo-Nirenberg inequalities}

In two papers,~\cite{MR102740,MR109940}, E.~Gagliardo and L.~Nirenberg established a number of interpolation inequalities which are now famous. Some historical details can be found in~\cite{Nirenberg_2020}. Here we are interested in the sub-family
\be{GNS}
\nrm{\nabla f}2^\theta\,\nrm f{p+1}^{1-\theta}\ge\mathcal C_{\mathrm{GNS}}(p)\,\nrm f{2p}
\ee
where the exponent $\theta=\frac{d\,(p-1)}{(d+2-p\,(d-2))\,p}$ is fixed by scaling considerations and functions $f$ are taken in the space obtained as the completion of the space of infinitely differentiable functions on $\R^d$ with compact support, with respect to the norm $f\mapsto(1-\theta)\,\nrm f{p+1}+\theta\,\nrm{\nabla f}2$. We shall say that the exponent $p$ is an \emph{admissible exponent} if
\[
p\in(1,+\infty)\quad\mbox{if}\quad d=1\;\mbox{or}\;2\,,\quad p\in(1,p^\star]\quad\mbox{if}\quad d\ge3\quad\mbox{with}\quad p^\star:=\tfrac d{d-2}\,.
\]
In the limit case where $p=p^\star$, $d\ge3$, for which $\theta=1$, we notice that $2\,p^\star=2^*$ and recover~\eqref{SobolevRd} as a limit case of~\eqref{GNS}. Other interesting limits can be obtained: as $p\to1$, Inequality~\eqref{GNS} degenerates according to~\cite{DelPino2002} into the Euclidean logarithmic Sobolev inequality in scale invariant form, while in the limit as $p\to+\infty$ when $d=2$ (see for instance~\cite{del_Pino_2012}), we recover the \emph{Euclidean Onofri inequality}. See~\cite[Chapter~1]{BDNS2021} for more details. Inequality~\eqref{GNS} was proved in~\cite{DelPino2002}. Actually, in~\cite{MR1112572}, an earlier result was stated with a rather complete scheme of a proof, up to the delicate question of the uniqueness. None of the authors of~\cite{DelPino2002} were aware of~\cite{MR1112572} but they had by the time the paper was written fine uniqueness results available from the mathematical literature. Entropy methods are now providing alternative tools which have been exploited in~\cite[Chapter~1]{BDNS2021}. In any case, the key issue is that optimal functions in~\eqref{GNS} can be identified. They are given up to dilations, translations and multiplications by a constant by
\be{Aubin.Talenti}
\mathsf g(x):=\(1+|x|^2\)^{-\frac1{p-1}}\quad\forall\,x\in\R^d\,.
\ee
Hence they generalize the \emph{Aubin-Talenti functions}. For this reason, we shall still use this denomination and denote again by $\mathfrak M$ the corresponding manifold. Actually there are deeper reasons for that than the algebraic expression of $\mathsf g$, that were clarified in~~\cite[Section~6.10]{MR3155209} (also see~\cite[Section~1.3.1]{BDNS2021}). Stability results for~\eqref{GNS} is one of the main achievement of entropy methods based on nonlinear diffusions and will detailed in Section~\ref{Sec:FDE} (see Theorems~\ref{Thm:stabilityDraft2} and~\ref{Thm:Main}).

%.......................................................................
\subsubsection{Caffarelli-Kohn-Nirenberg inequalities}

The critical inequality
\be{CKN}
\(\ird{\frac{|f|^p}{|x|^{b\,p}}}\)^{2/p}\le\,\mathsf C_{a,b}\ird{\frac{|\nabla f|^2}{|x|^{2\,a}}}
\ee
is valid for any function $f\in\left\{\,f\in\mathrm L^p\big(\R^d,|x|^{-b}\,dx\big)\,:\,|x|^{-a}\,|\nabla f|\in\mathrm L^2\big(\R^d,dx\big)\right\}$ under the conditions that $a\le b\le a+1$ if $d\ge3$, $a<b\le a+1$ if $d=2$, $a+1/2<b\le a+1$ if $d=1$, and $a<a_c:=(d-2)/2$ where the exponent $p=2\,d/(d-2+2\,(b-a))$ is again fixed by scaling considerations. Here $\mathsf C_{a,b}$ denotes the optimal constant. Inequality~\eqref{CKN} was apparently introduced first by V.P.~Il'in in~\cite{Ilyin} but is more known in the literature as the \emph{critical Caffarelli-Kohn-Nirenberg inequalities}, according to~\cite{Caffarelli1984}. An important step in the understanding of this inequality came with~\cite{Catrina2001} where the question of the \emph{symmetry breaking} problem was raised: in~\eqref{CKN}, is optimality achieved among radially symmetric functions ? Notice that partial symmetry results were known before, see, \emph{e.g.}, \cite{Chou-Chu-93,MR1731336}: the point in~\cite{Catrina2001} (also see~\cite{MR2001882}) was to provide a mechanism for proving that symmetry breaking also occurs. If symmetry occurs, it can be proved that the equality case is achieved by
\be{g-CKN}
\mathsf g(x):=\(1+|x|^{(p-2)(a_c-a)}\)^{-\frac2{p-2}}\quad\forall\,x\in\R^d
\ee
and the optimal constant is then given by
\[
\mathsf C_{a,b}=\tfrac2p\,((a_c-a)^{-\frac{p+2}p}\(\frac{(p-2)\,\Gamma\big(\frac{3\,p-2}{2\,(p-2)}\big)}{2\,\sqrt\pi\,\big|\mathbb S^{d-1}\big|\;\Gamma\big(\frac p{p-2}\big)}\)^\frac{p-2}p\,.
\]
Assuming that this was the case, V.~Felli and M.~Schneider systematized the strategy of~\cite{Catrina2001} in~\cite{Felli2003}: by linearizing the functional made of the difference of the two sides in~\eqref{CKN}, they found a negative eigenvalue if $b<b_{\rm FS}(a)$ with
\[
b_{\rm FS}(a):=\frac{d\,(a_c-a)}{2\,\sqrt{(a_c-a)^2+d-1}}+a-a_c\,,
\]
thus proving the linear instability of $\mathsf g$, a contradiction. As a consequence, \emph{symmetry breaking} occurs if $b<b_{\rm FS}(a)$. Reciprocally, one of the major achievements of entropy methods is the following symmetry result.
%-----------------------------------------------------------------------
\begin{thm}[Symmetry \emph{vs.}~symmetry breaking, critical case,~\cite{Dolbeault2016}]\label{Thm:CKN} Assume that $d\ge2$, $a<a_c$, and $b_{\rm FS}(a)\le b\le a+1$ if $a<0$. Then equality in~\eqref{CKN} is achieved if and only if, up to a scaling and a multiplication by a constant, $f=\mathsf g$.\end{thm}
%-----------------------------------------------------------------------
The result in~\cite{Dolbeault2016} is actually stronger: it is proved that there is no other critical point of the Euler-Lagrange equation, which is a \emph{rigidity result}. The method is not limited to critical inequalities and there is also a family of \emph{subcritical Caffa\-relli-Kohn-Nirenberg inequalities}, exactly as in the case without weights. Similar results have been obtained in~\cite{Bonforte2017a,Bonforte2017,Dolbeault2017} for the inequality
\be{CKNsub}
\(\ird{\frac{|f|^{2\,p}}{|x|^\gamma}}\)^{1/p}\le\,\mathsf C_{\beta,\gamma,p}\(\ird{\frac{|\nabla f|^2}{|x|^\beta}}\)^\theta\(\ird{\frac{|f|^{p+1}}{|x|^\gamma}}\)^{1-\theta}
\ee
with optimal constant $\mathsf C_{\beta,\gamma,p}$, parameters $\beta$, $\gamma$ and $p$ such that
\be{parameters}
\gamma-2<\beta<\frac{d-2}d\,\gamma\,,\quad\gamma\in(-\infty,d)\,,\quad p\in\(1,p_\star\right]\quad\mbox{with}\quad p_\star:=\frac{d-\gamma}{d-\beta-2}\,,
\ee
and an exponent $\theta=\frac{(d-\gamma)\,(p-1)}{p\,(d+\beta+2-2\,\gamma-p\,(d-\beta-2))}$ which is determiend by scaling considerations. The admissible set of parameters $(\beta,\gamma)$ is restricted by the condition $p\le p_\star$, which means
\[
\beta\ge d-2+\frac{\gamma-d}p\,.
\]
If symmetry holds, then the equality case in~\eqref{CKNsub} is achieved by
\[
\mathsf g(x):=\(1+|x|^{2+\beta-\gamma}\)^{-\frac1{p-1}}\quad\forall\,x\in\R^d\,.
\]
The following results are taken from~\cite[Theorem~2]{Bonforte2017a} and in~\cite[Theorem~1.1]{Dolbeault2017}.
%-----------------------------------------------------------------------
\begin{thm}[Symmetry \emph{vs.}~symmetry breaking, critical case,~\cite{Bonforte2017a,Dolbeault2017}] Assume that $d\ge2$ and $\beta$, $\gamma$ and $p$ satisfy~\eqref{parameters}. Symmetry holds in~\eqref{CKNsub} if and only if
\[
\gamma<d\,,\quad\gamma-2<\beta<\frac{d-2}d\,\gamma\quad\mbox{and}\quad\beta\le\beta_{\rm FS}(\gamma):=d-2-\sqrt{(\gamma-d)^2-4\,(d-1)}\,.
\]
\end{thm}
%-----------------------------------------------------------------------

%%%%%%%%%%%%%%%%%%%%%%%%%%%%%%%%%%%%%%%%%%%%%%%%%%%%%%%%%%%%%%%%%%%%%%%%
\subsection{Branches of solutions}\label{Sec:Branches}

In 1983, H. Br\'ezis and L.~Nirenberg published a paper,~\cite{MR709644}, devoted to the positive solutions of the equation
\be{BNeqn}
-\Delta f+\lambda\,f=f^{p-1}
\ee
in bounded domains, with homogeneous Dirichlet boundary conditions. They were interested in the case of the critical exponent $p=2^*$, $d\ge3$, and related issues of compactness in connection with the Yamabe problem. This paper had a considerable impact. With~\cite{MR867665}, it was a starting point for an amazing effort to construct various solutions by variational methods or using Lyapunov-Schmidt reduction techniques, not only on the Euclidean space but also on Riemannian manifolds (where $\Delta$ has to be understood as the Laplace-Beltrami operator). One of the typical issues is to understand how branches of solutions depend on the parameter $\lambda$ and at which values bifurcations may eventually occur, as part of the larger question of the multiplicity of the solutions and their classification. Non-existence (typically based on Poho\v zaev's identity) and uniqueness results are another related issue, with a series of prominent papers, 
\cite{MR333489,MR615628,MR1134481}, that can now be reinterpreted using entropy methods. In~\cite{MR1134481} for instance, it is proved that~\eqref{BNeqn} written on the sphere $\mathbb S^d$ has a \emph{unique} positive solution, $f\equiv\lambda^{-1/(p-2)}$, if and only if $\lambda\le d/(p-2)$. As a consequence, according to~\cite[Corollary~6.1]{MR1134481} (also see~\cite{MR1230930}), we obtain the inequality 
\be{Ineq:GNS}
\nrmSd{\nabla f}2^2\ge\frac d{p-2}\(\nrmSd fp^2-\nrmSd f2^2\)\quad\forall\,f\in\mathrm H^1(\mathbb S^d,d\mu)
\ee
where $d\mu=\big|\mathbb S^d\big|^{-1}\,dv_g$ denotes the uniform probability measure on the sphere and $p\in(2,2^*)$, with the convention that $2^*=+\infty$ if $d=1$ or $2$. Inequality~\eqref{Ineq:GNS} has been studied in~\cite{MR1134481} by \emph{rigidity} methods, in~\cite{MR1230930} by techniques of harmonic analysis, and using the \emph{carr\'e du champ} method in~\cite{MR1231419,MR1412446,Demange_2008}, for any $p>2$, with improvements and extensions in~\cite{DEKL,1504}. We refer to~\cite{1703} for a review and~\cite{Dolbeault_2020} for most recent results.

If $d\ge3$, the case $p=2^*$ is also covered: Inequality~\eqref{Ineq:GNS} is the Sobolev inequality on the sphere and it is equivalent to~\eqref{SobolevRd}, up to a stereographic projection. Taking into account the normalization of $d\mu$, it is easy to recover the expression of $\mathsf S_d$. Another limit case is the \emph{logarithmic Sobolev inequality} on $\mathbb S^d$,
\be{Ineq:logSob}
\nrmSd{\nabla f}2^2\ge\frac d2\,\iSd{|f|^2\,\log\(\frac{|f|^2}{\nrmSd f2^2}\)}\quad\forall\,f\in\mathrm H^1(\mathbb S^d,d\mu)\setminus\{0\}\,,
\ee
which corresponds to the limit as $p\to2_+$. Remarkably,~\eqref{Ineq:GNS} is also valid for any $p\in[1,2)$ where the case $p=1$ corresponds simply to the Poincar\'e inequality on the sphere. In this range, and actually for any $p\in[1,2)\cup(2,+\infty)$ if $d=1$ or $p\in[1,2)\cup(2,2^\#]$ if $d\ge2$, where 
\[
2^\#:=\frac{2\,d^2+1}{(d-1)^2}
\]
is the \emph{Bakry-Emery exponent}, Inequalities~\eqref{Ineq:GNS} and~\eqref{Ineq:logSob} were proved earlier by the \emph{carr\'e du champ} method of D.~Bakry and M.~Emery in~\cite{MR772092}.

{}From the point of view of elliptic PDEs, the \emph{carr\'e du champ} method amounts to test~\eqref{BNeqn} by $\Delta f+(p-1)\,|\nabla f|^2/f$. After using the Bochner-Lichne\-rowicz-Weitzenb\"ock formula
\be{BLW}
\tfrac12\,\Delta\,(|\nabla f|^2)=|\mathrm {Hess}f|^2+\nabla\cdot(\Delta f)\cdot\nabla f+\mathrm{Ric}(\nabla f,\nabla f)\,,
\ee
the fact that $\mathrm{Ric}(\nabla f,\nabla f)=(d-1)\,|\nabla f|^2$ on $\mathbb S^d$, and a few integrations by parts, we obtain the identity
\[
\iSd{\(\tfrac d{d-1}\,|Q_f|^2+(d-\lambda)\,|\nabla f|^2+\gamma(p)\,\frac{|\nabla f|^4}{f^2}\)}=0
\]
with
\be{Qf}
Q_f:=\mathrm{Hess}f-\frac1d\,\Delta f\,\mathrm{Id}-\frac{d-1}{d+2}\,(p-1)\(\frac{\nabla f\otimes\nabla f}f-\frac1d\,\frac{|\nabla f|^2}f\,\mathrm{Id}\)
\ee
and
\be{gamma(p)}
\gamma(p):=\(\frac{d-1}{d+2}\)^2(p-1)\,(2^\#-p)\,.
\ee
Here we use the notation $|\mathsf A|^2=\sum_{i,j=1}^d\mathsf A_{ij}^2$ if $\mathsf A=(\mathsf A_{ij})_{i,j=1}^d$ is a $d\times d$ matrix. It is therefore clear that the unique positive solution of~\eqref{BNeqn} is a constant if $\lambda\le d$. This result has an interesting consequence in terms of branches of solutions. On $\mathrm H^1(\mathbb S^d,d\mu)$, the functional 
\[
f\mapsto\nrmSd{\nabla f}2^2-\frac\lambda{p-2}\(\nrmSd fp^2-\nrmSd f2^2\)
\]
always admits $f\equiv1$ as a critical point, there is no other positive solution of~\eqref{BNeqn} if $\lambda\le d$ and the functional is nonnegative. For any $\lambda>d$, a simple linearization around the constant solution shows that there is a minimizer which makes the functional negative and $\lambda=d$, which corresponds to the first positive eigenvalue of the Laplace-Beltrami operator, is a bifurcation point. Let us define
\[
\mu(\lambda):=\inf_{u\in\mathrm H^1(\S^d)\setminus\{0\}}\frac{(p-2)\,\nrmSd{\nabla f}2^2+\lambda\,\nrmSd f2^2}{\nrmSd fp^2}\,.
\]
The uniqueness result in the range $\lambda\in[0,d]$ and the bifurcation result at $\lambda=d$ have the following consequences. If $\lambda\le d$, then $\mu(\lambda)=\lambda$, while $\mu(\lambda)<\lambda$ if $\lambda>d$. See Fig.~\ref{F1}.
%-----------------------------------------------------------------------
\setlength\unitlength{1cm}
\begin{figure}[ht]
\begin{picture}(12,4)
\put(3,0){\includegraphics[width=6cm]{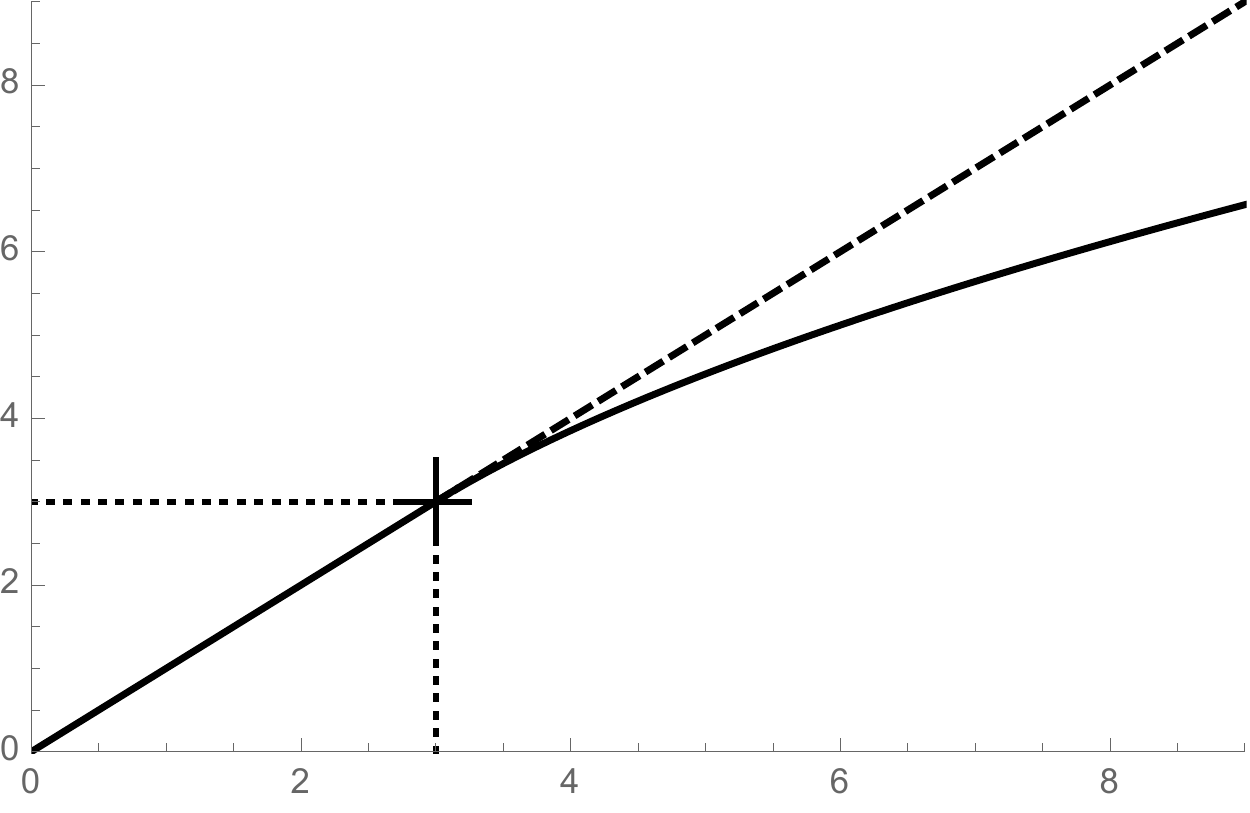}}
\put(8.8,0.5){$\lambda$}
\put(3.4,3.7){$\mu$}
\put(7,1.9){$\mu=\mu(\lambda)$}
\put(7.25,3.5){$\mu=\lambda$}
\end{picture}
\caption{\label{F1}Numerical computation of the bifurcation diagram for $d=3$ and $p=3$. The branch emerging at $\lambda=d$ from the straight line $\mu=\lambda$ of the constant solutions is expected to realize the infimum $\lambda\mapsto\mu(\lambda)$.}
\end{figure}
%-----------------------------------------------------------------------

The carr\'e du champ method is limited to the range $p\le2^\#$, but the cases $p\in(2^\#,+\infty)$ if $d=1$ or~$2$, or $p\in(2^\#,2^*]$ if $d\ge3$ can be dealt with using another test function, as we shall see below: see~\cite{MR1134481,Dolbeault_2014,1504} and references therein for more details. The method admits various extensions, for instance to smooth compact connected Riemannian manifolds, or improvements which can be used in order to obtain stability estimates. For instance, on $\S^d$, the inequality
\be{improvedineq}
\nrmSd{\nabla f}2^2\ge\frac d{2-p-\gamma}\(\nrmSd f2^2-\nrmSd fp^{2-\frac{2\,\gamma}{2-p}}\,\nrmSd f2^{\frac{2\,\gamma}{2-p}}\)\quad\forall\,f\in\mathrm H^1(\mathbb S^d)
\ee
is established in~\cite{Dolbeault_2020} for $d\ge1$, $p\in[1,2)\cup(2,2^\#)$ with the convention $2^\#=+\infty$ if $d=1$ and $\gamma=\gamma(p)$ given by~\eqref{gamma(p)} if $d\ge3$ and $\gamma(p)=(p-1)/3$ if $d=1$, under the condition that $\gamma(p)\neq2-p$. If $\gamma(p)=2-p$, there is a logarithmic version of the inequality. Inequality~\eqref{improvedineq} can be used to obtain estimates on the branches of solutions as in~\cite[Theorem~2]{Dolbeault_2020}.

A major difficulty in the elliptic point of view of~\cite{MR333489,MR615628} is the construction of the appropriate test functions. Reinterpreted in terms of adapted entropy methods and gradient flow estimates, building the estimates becomes a much simpler task. This is what we will explain next.

%%%%%%%%%%%%%%%%%%%%%%%%%%%%%%%%%%%%%%%%%%%%%%%%%%%%%%%%%%%%%%%%%%%%%%%%
\subsection{Entropies and \emph{carr\'e du champ} methods}

\emph{Carr\'e du champ} methods were introduced in~\cite{Bakry1985} by D.~Bakry and M.~Emery and have been successfully applied to a wide range of questions: a general overview can be found in~\cite{MR3155209}. Here we give two simple examples, first in the classical setting of linear diffusions with a confinement potential on $\R^d$, and then we revisit the inequalities~\eqref{Ineq:GNS} of Section~\ref{Sec:Branches} dealing with interpolation on $\mathbb S^d$.

%.......................................................................
\subsubsection{Inequalities, rates of convergence and entropy methods for linear diffusions}\label{Sec:LinearFlows}

On $\R^+\times\R^d$, let us consider the \emph{heat equation}
\be{Heat}
\frac{\partial u}{\partial t}=\Delta u\,,\quad u(t=0,x)=u_0(x)
\ee
and assume that $u_0\in\mathrm L^1\cap\mathrm L^2(\R^d)$ is nonnegative. Among various classical results, it is well known that as $t\to+\infty$, the solution behaves like the Green function $G(t,x)=(4\,\pi\,t)^{-d/2}\,\exp(-|x|^2/(4\,t))$. As a consequence, we know that $\nrm{u(t,\cdot)}2^2=O(t^{-d/2})$ as $t\to+\infty$, which is easily recovered using Nash's inequality:
\be{Nash}
\nrm u2^{2+\frac4d}\le\mathcal C_{\rm Nash}\,\nrm u1^\frac4d\,\nrm{\nabla u}2^2
\ee
For a solution of~\eqref{Heat}, we know that $\nrm{u(t,\cdot)}1=\nrm{u_0}1=:M$ for any $t\ge0$ and can indeed write that
\[
\frac d{dt}\nrm{u(t,\cdot)}2^2=-\,2\,\nrm{\nabla u(t,\cdot)}2^2\le-\,\frac2{\mathcal C_{\rm Nash}}\,\frac{\nrm{u(t,\cdot)}2^{2+4/d}}{M^{4/d}}
\]
and the decay estimate easily follows after integrating with respect to $t$. A slightly more delicate issue is to obtain an \emph{intermediate asymptotics} estimate, for instance an estimate of $u(t,\cdot)-M\,G(t,\cdot)$ in $\mathrm L^1(\R^d)$. This can be done with entropy methods as follows. Using the change of variables
\[
u(t,x)=R(t)^{-d}\,v\(\log R(t),\tfrac x{R(t)}\)\quad\mbox{with}\quad R(t)=\sqrt{1+2\,t}\quad\forall\,(t,x)\in\R^+\times\R^d\,,
\]
solving~\eqref{Heat} is equivalent to solve the \emph{Fokker-Planck equation}
\be{Fokker-Planck}
\frac{\partial v}{\partial t}=\Delta v+\nabla\cdot(x\,v)\,,\quad v(t=0,x)=u_0(x)\,.
\ee
A remarkable fact is that~\eqref{Fokker-Planck} admits a stationary solution $M\,\mu$ with $\mu(x):=G(1/2,x)$. Up to an additional change of unknown function, the function $w(t,x)=v(t,x)/\mu(x)$ solves the \emph{Ornstein-Uhlenbeck equation}
\be{Ornstein-Uhlenbeck}
\frac{\partial w}{\partial t}=\Delta w+x\cdot\nabla w\,,\quad w(t=0,x)=\frac{u_0(x)}{\mu(x)}\,.
\ee
With the Gaussian measure $d\mu=\mu(x)\,dx$, an elementary computations shows that
\[
\frac d{dt}\irdmu{\left|w-M\right|^2}=-\,2\irdmu{\left|\nabla w\right|^2}\le-\,2\irdmu{\left|w-M\right|^2}
\]
where the last inequality is a simple consequence of the \emph{Gaussian Poincar\'e inequality}
\[
\irdmu{\left|f-\bar f\right|^2}\le\irdmu{|\nabla f|^2}\quad\mbox{where}\quad\bar f:=\irdmu f\,,\quad\forall\,f\in\mathrm H^1(\R^d)\,.
\]
It seems that this inequality was explicitly written in that form in~\cite{Nash58}, in the very same paper in which J.~Nash proved~\eqref{Nash}. With this observation and a Gr\"onwall estimate, we obtain the exponential decay of $\irdmu{\left|w-M\right|^2}\ge\(\irdmu{\left|w-M\right|}\)^2$, which controls $\nrm{u(t,\cdot)-M\,G(t,\cdot)}1$. Alternatively, one can directly consider the \emph{entropy} and compute
\[
\frac d{dt}\irdmu{w\,\log(w/M)}=-\irdmu{w\,|\nabla\log w|^2}\le-\,2\irdmu{w\,\log(w/M)}
\]
where the last inequality is a simple consequence of the \emph{Gaussian logarithmic Sobolev inequality}
\[
\irdmu{f^2\,\log(f^2/M)}\le2\irdmu{|\nabla f|^2}\quad\mbox{where}\quad M=\irdmu{f^2}\,,\quad\forall\,f\in\mathrm H^1(\R^d)\,.
\]
Using again a Gr\"onwall estimate and the Pinsker-Csisz\'ar-Kullback inequality
\[
\irdmu{w\,\log(w/M)}\ge\frac 1{4\,M}\(\irdmu{|w-M|}\)^2\,,
\]
we can conclude as before. It turns out that the heat equation can be seen as the gradient flow of the entropy $\ird{u\,\log u}$ with respect to Wasserstein's distance as shown in~\cite{MR1617171} and the entropy formulation turns out to be better adapted when linear diffusion are replaced by nonlinear diffusions. For more details on Nash's inequality, consequences for the heat equation, the comparison of the decay estimates and more detailed references, we refer to~\cite{BDNS2021,Bouin_2020}.

\medskip A natural extension of the above heat related estimates is to consider linear diffusions with confinement due to a given potential $\phi(x)$, which amounts to study the \emph{Ornstein-Uhlenbeck equation} with potential $\phi$, that is,
\be{Ornstein-Uhlenbeck2}
\frac{\partial w}{\partial t}=\Delta w-\nabla\phi\cdot\nabla w\,.
\ee
This equation is reduced to~\eqref{Ornstein-Uhlenbeck} if $\phi(x)=|x|^2/2$. The invariant measure is now $d\mu=Z^{-1}\,e^{-\phi}\,dx$ with $Z=\ird{e^{-\phi}}$ and the \emph{Ornstein-Uhlenbeck operator} $\mathcal L:=\Delta-\nabla\phi\cdot\nabla$ is self-adjoint on $\mathrm L^2(\R^d,d\mu)$ in the sense that
\[
\irdmu{w_1\,\mathcal Lw_2}=-\irdmu{\nabla w_1\cdot\nabla w_2}
\]
As seen above, obtaining convergence rates heavily relies on functional inequalities. One of the key features of entropy methods is that the flow of~\eqref{Ornstein-Uhlenbeck2} can be used to establish these inequalities by the \emph{carr\'e du champ} method, under the assumption that
\be{BLW1}
\mathrm{Hess}\,\phi\ge\kappa\,\mathrm{Id}\quad\mbox{a.e.}
\ee
for some $\kappa>0$. Under this assumption, we claim that
\be{Ineq0}
\irdmu{|\nabla f|^2}\ge\frac\kappa{p-2}\(\(\irdmu{|f|^p}\)^{2/p}-\irdmu{|f|^2}\)\quad\forall\,f\in\mathrm H^1(\R^d,d\mu)
\ee
for any $p\in[1,2)$, and also, by taking the limit as $p\to2_-$ and with the notation $M=\irdmu{|f|^2}$,
\[
\irdmu{|\nabla f|^2}\ge\frac\kappa2\irdmu{|f|^2\,\log\(|f|^2/M\)}\quad\forall\,f\in\mathrm H^1(\R^d,d\mu)\,.
\]
If $w$ solves~\eqref{Ornstein-Uhlenbeck2} with initial datum $w_0=|f|^p$, the strategy of the \emph{carr\'e du champ} method is to notice that $\frac d{dt}\irdmu w=0$ and that $w^2$ converges to the constant $M=\irdmu{|f|^p}$ while
\[
\frac d{dt}\(\irdmu{\left|\nabla w^{1/p}\right|^2}+\frac\kappa{p-2}\irdmu{w^{2/p}}\)\le0\,,
\]
so that
\begin{multline*}
\irdmu{|\nabla f|^2}-\frac\kappa{p-2}\(\(\irdmu{|f|^p}\)^{2/p}-\irdmu{|f|^2}\)\\
\ge\lim_{t\to+\infty}\irdmu{|\nabla w(t,x)^{1/p}|^2}-\frac\kappa{p-2}\(\(\irdmu{w(t,x)^{2/p}}\)^{2/p}-M^{2/p}\)=0\,.
\end{multline*}
The proof relies on a commutation of the derivatives which is is similar to the Bochner-Lichnerowicz-Weitzenb\"ock formula~\eqref{BLW}. Performing a few integrations by parts, we compute for some function $h\in\mathrm H^2(\R^d,d\mu)$ the integral
\begin{align*}
\irdmu{(\mathcal Lh)^2}&=-\irdmu{\nabla h\cdot\nabla(\mathcal Lh)}\\
&=-\irdmu{\nabla h\cdot(\mathcal L\nabla h)}+\irdmu{\nabla h\cdot[\mathcal L,\nabla]\,h}\\
&=\irdmu{|\mathrm{Hess}h|^2}+\irdmu{\nabla h\cdot[\mathcal L,\nabla]\,h}\,.
\end{align*}
Since $\nabla h\cdot[\mathcal L,\nabla]\,h=\mathrm{Hess}\,\phi:\nabla h\otimes\nabla h\ge\kappa\,|\nabla h|^2$ according to~\eqref{BLW1}, this proves
\be{Id1-OU}
\irdmu{(\mathcal Lh)^2}\ge\irdmu{|\mathrm{Hess}h|^2}+\kappa\irdmu{|\nabla h|^2}\,.
\ee
Another elementary identity is given by an integration by parts:
\be{Id2-OU}
\irdmu{(\mathcal Lh)\,\frac{|\nabla h|^2}h}=\irdmu{\frac{|\nabla h|^4}{h^2}}-2\irdmu{\mathrm{Hess}h:\frac{\nabla h\otimes\nabla h}h}\,.
\ee
Hence with $h=w^{1/p}$ and $w$ solving~\eqref{Ornstein-Uhlenbeck2}, so that $h$ solves
\[
\frac{\partial h}{\partial t}=\mathcal Lh+(p-1)\,\frac{|\nabla h|^2}h
\]
with initial datum $h(t=0,\cdot)=|f|$, we find that
\[
\frac d{dt}\irdmu{w^{2/p}}=\frac d{dt}\irdmu{h^2}=-\,2\,(2-p)\irdmu{|\nabla h|^2}\,,
\]
\[
\frac d{dt}\irdmu{\left|\nabla w^{1/p}\right|^2}=\frac d{dt}\irdmu{\left|\nabla h\right|^2}=-\,2\irdmu{(\mathcal Lh)\(\mathcal Lh+(p-1)\,\frac{|\nabla h|^2}h\)}
\]
Altogether,
\begin{multline*}
\frac d{dt}\(\irdmu{\left|\nabla w^{1/p}\right|^2}+\frac\kappa{p-2}\irdmu{w^{2/p}}\)\\
=-\,2\irdmu{\left|\mathrm{Hess}h-(p-1)\,\frac{\nabla h\otimes\nabla h}h\right|^2}-\,2\,(p-1)\,(2-p)\irdmu{\frac{|\nabla h|^4}{h^2}}\le0\,,
\end{multline*}
which completes the proof of~\eqref{Ineq0}. We refer to \cite{MR2081075,MR2435196,doi:10.1142/S0218202518500574} for further details.

Condition~\eqref{BLW1} is the standard criterion in the \emph{carr\'e du champ} method of D.~Bakry and M.~Emery, but it is a pointwise condition which is by far too strong, as we use it only after an integration. See~\cite{MR2435196} for some considerations on non-local, spectral conditions. There is actually a deeper property, for the following reason. Let us assume that $\kappa_1>0$ is the optimal constant in
\be{Ineq1}
\irdmu{|\nabla f|^2}\ge\frac{\kappa_1}{p-2}\(\(\irdmu{|f|^p}\)^{2/p}-\irdmu{|f|^2}\)\quad\forall\,f\in\mathrm H^1(\R^d,d\mu)
\ee
if $p\in[1,2)$, and in
\be{Ineq2}
\irdmu{|\nabla f|^2}\ge\frac{\kappa_1}2\irdmu{|f|^2\,\log\(|f|^2/M\)}\quad\forall\,f\in\mathrm H^1(\R^d,d\mu)
\ee
if $p=2$, with $M=\irdmu{|f|^2}$. Now, let us assume that $\kappa_2>0$ is the optimal constant in
\begin{multline}\label{Ineq3}
\irdmu{\left|\mathrm{Hess}f-(p-1)\,\frac{\nabla f\otimes\nabla f}f\right|^2}+(p-1)\,(2-p)\irdmu{\frac{|\nabla f|^4}{f^2}}\\
+\irdmu{\mathrm{Hess}\,\phi:\nabla f\otimes\nabla f}\ge\kappa_2\irdmu{|\nabla f|^2}\quad\forall\,f\in\mathrm H^2(\R^d,d\mu)\,.
\end{multline}
Finally, let us consider the Poincar\'e inequality
\be{Ineq4}
\irdmu{|\nabla f|^2}\ge\kappa_0\(\irdmu{|f|^2}-\(\irdmu f\)^2\)\quad\forall\,f\in\mathrm H^1(\R^d,d\mu)
\ee
and assume that $\kappa_0>0$ is the optimal constant. %-----------------------------------------------------------------------
\begin{prop}\label{Prop:New1} Let $d\ge1$, $p\in[1,2)$, $e^{-\phi}\in\mathrm L^1(\R^d)$, and assume that the probability measure $d\mu=Z^{-1}\,e^{-\phi}\,dx$ with $Z=\ird{e^{-\phi}}$ is such that~\eqref{Ineq4} holds for some $\kappa_0>0$. With the above notations, we have $\kappa_1>0$ and $\kappa_2\le\kappa_1\le\kappa_0$. Moreover, if $\kappa_2=\kappa_1$, then $\kappa_2=\kappa_1=\kappa_0$.\end{prop}
%-----------------------------------------------------------------------
\begin{proof} By H\"older's inequality, we know that $\(\irdmu{|f|}\)^2\le\(\irdmu{|f|^p}\)^{2/p}$, so that $\kappa_1\ge\kappa_0\,(2-p)>0$. 

The fact that $\kappa_2\le\kappa_1$ is part of the standard carr\'e du champ method and has already been observed above: we know that \eqref{Ineq1} and~\eqref{Ineq2} are consequences of~\eqref{Ineq3}. Inequality~\eqref{Ineq1} or~\eqref{Ineq2} applied to the function $h=1+\varepsilon\,f$ gives, at leading order in the limit as $\varepsilon\to0$, the Poincar\'e inequality
\[
\irdmu{|\nabla f|^2}\ge\kappa_1\(\irdmu{|f|^2}-\(\irdmu f\)^2\)\quad\forall\,f\in\mathrm H^1(\R^d,d\mu)\,,
\]
where the constant $\kappa_1$ is not necessarily optimal: $\kappa_1\le\kappa_0$.

If $p=1$, we have $\kappa_2=\kappa_1=\kappa_0$ which easily follows from an expansion of the square: let $f\in\mathrm H^1(d\mu)$ be such that $\irdmu f=0$. Using the fact that $\irdmu{(\mathcal Lf+\kappa_0\,f)^2}\ge0$, we find that
\begin{multline*}
\irdmu{\left|\mathrm{Hess}f\right|^2}+\irdmu{\mathrm{Hess}\,\phi:\nabla f\otimes\nabla f}-\kappa_0\irdmu{|\nabla f|^2}\\
\ge\kappa_0\(\irdmu{|\nabla f|^2}-\kappa_0\irdmu{|f|^2}\)\ge0\,,
\end{multline*}
which proves that $\kappa_2\ge\kappa_0$. In the remainder of the proof, we assume that $p\in(1,2)$.

In the equality case $\kappa_2=\kappa_1$, we read from the previous computations that there is no non-constant optimal function for~\eqref{Ineq1}. This fact is easily obtained as follows: if $f$ is an optimal function for~\eqref{Ineq1}, let us take $w_0=|f|^p$ and consider the solution~$w$ of~\eqref{Ornstein-Uhlenbeck2} with initial datum $w_0$. At $t=0$, we obtain $\irdmu{f^{-2}\,|\nabla f|^4}=0$, a contradiction. Hence, if we consider a minimizing sequence $(f_n)_{n\in\N}$ for
\[
\kappa_1=\inf\mathcal Q_p[f]\quad\mbox{where}\quad\mathcal Q_p[f]:=\frac{(2-p)\irdmu{|\nabla f|^2}}{\irdmu{|f|^2}-\(\irdmu{|f|^p}\)^{2/p}}\,,
\]
we can write that $f_n=1+\varepsilon_n\,h_n$ for any $n\in\N$, with $\lim_{n\to+\infty}\varepsilon_n=0$ and $(h_n)_{n\in\N}$ uniformly bounded in $\mathrm H^1(\R^d\,d\mu)$, so that
\[
\kappa_0\ge\kappa_1=\lim_{n\to+\infty}\mathcal Q_p[1+\varepsilon_n\,h_n]=\lim_{n\to+\infty}\mathcal Q_1[h_n]\ge\kappa_0\,.
\]
For further details in a similar problem, see~\cite{BDNS2021}.
\end{proof}

It is possible to prove various other estimates. For instance, the function
\[
p\mapsto\frac{2\,p}{2-p}\(\irdmu{|f|^2}-\(\irdmu{|f|^p}\)^{2/p}\)
\]
is non-decreasing according to~\cite[Lemma~1]{MR1796718}, so that we have the bounds
\[
(2-p)\,\kappa_0\le\kappa_1\le p\,\kappa_0\,.
\]
If the logarithmic Sobolev inequality~\eqref{Ineq2} holds for some $\kappa_1>0$, then the Poincar\'e inequality~\eqref{Ineq4} holds for some $\kappa_0\ge\kappa_1$ and~\eqref{Ineq1} holds for some constant which can be estimated in terms of $\kappa_0$ and $\kappa_1$. Additionally, if $\kappa_0=\kappa_1$, then $\kappa_0$ is also the optimal constant in~\eqref{Ineq1} and does not depend on $p\in[1,2]$: see~\cite{MR1796718,0528}. This occurs if $\phi(x)=|x|^2/2$: in that case, we have $\kappa_2=\kappa_1=\kappa_0$, for any $p\in[1,2]$. 

The \emph{carr\'e du champ} method gives only a sufficient condition to establish~\eqref{Ineq1} or~\eqref{Ineq2}, and the constant $\kappa$ in~\eqref{BLW1} is then a lower estimate for $\kappa_2$. In some cases, it can be proved by a direct computation that $\kappa=\kappa_2$ and then \emph{all optimal constants} in~\eqref{Ineq1}-\eqref{Ineq4} are known: this is the case for the harmonic potential and, as we shall see next, we have a similar picture in all cases of interest considered in this paper. For this reason, the method provides sharp results. From the point of view of the flow, if $w$ solves~\eqref{Ornstein-Uhlenbeck2} with initial datum $w_0=|f|^p$, for some non-constant function $f\in\mathrm H^1(\R^d,d\mu)$, then we have
\[
\mathcal Q_p[f]\ge\lim_{t\to+\infty}\mathcal Q_p\left[w(t,\cdot)^{1/p}\right]=\lim_{t\to+\infty}\mathcal Q_1\left[w(t,\cdot)^{1/p}\right]\ge\kappa_0\,.
\]
Then $\kappa_1=\inf\mathcal Q_p[f]=\kappa_0$ is recovered using $f=1+\varepsilon\,h$ as a test function and taking the limit as $\varepsilon\to0$, where $h$ is optimal for~\eqref{Ineq4}.

%.......................................................................
\subsubsection{Flows and interpolation inequalities on the sphere}\label{Sec:SphereFlows}

On the sphere $\mathbb S^d$, exactly the same method as for the Ornstein-Uhlenbeck equation~\eqref{Ornstein-Uhlenbeck2} applies. If we consider the heat equation
\[
\frac{\partial w}{\partial t}=\Delta w\,,\quad w(t=0,\cdot)=f^p\,,
\]
where $\Delta$ is the Laplace-Beltrami operator or, with $h=w^{1/p}$, the equation
\[
\frac{\partial h}{\partial t}=\Delta h+(p-1)\,\frac{|\nabla h|^2}h\,,\quad h(t=0,\cdot)=f\,,
\]
then we find that
\[
\frac d{dt}\iSd{h^p}=0\,,
\]
\[
\frac d{dt}\iSd{h^2}=-\,2\,(2-p)\iSd{|\nabla h|^2}\,,
\]
\[
\frac d{dt}\iSd{|\nabla h|^2}=-\,2\iSd{(\Delta h)\(\Delta h+(p-1)\,\frac{|\nabla h|^2}h\)}\,.
\]
The strategy of the computation is similar to the case of Equation~\eqref{Ornstein-Uhlenbeck2}. We replace~\eqref{BLW1} by the Bochner-Lichnerowicz-Weitzenb\"ock formula~\eqref{BLW} applied to~$h$. After an integration with respect to $d\mu$, this proves
\be{BLW2}
\iSd{(\Delta h)^2}=\iSd{|\mathrm{Hess}h|^2}+(d-1)\iSd{|\nabla h|^2}\,.
\ee
After a number integrations by parts, we obtain
\begin{multline*}
\frac d{dt}\left[\nrmSd{\nabla h}2^2-\frac\lambda{p-2}\(\nrmSd hp^2-\nrmSd h2^2\)\right]\\
=-\,2\iSd{\(\tfrac d{d-1}\,|Q_f|^2+(d-\lambda)\,|\nabla f|^2+\gamma(p)\,\frac{|\nabla f|^4}{f^2}\)}
\end{multline*}
with $Q_f$ and $\gamma(p)$ respectively given by~\eqref{Qf} and~\eqref{gamma(p)}. Using the fact that $h^p\to M$ as $t\to+\infty$, we find that $\nrmSd{\nabla h}2^2-\frac\lambda{p-2}\,\big(\nrmSd hp^2-\nrmSd h2^2\big)$ written with $\lambda=d$ is monotone non-increasing, with limit $0$ as $t\to+\infty$, so that~\eqref{Ineq:GNS} holds for the initial datum $h(t=0,\cdot)=f$. This method can be found in~\cite{Bakry1985} but we refer also to~\cite{1504} for details on the computations using the ultraspherical operator and to~\cite{DEKL,1504} for some improvements. It is however limited to 
$p\le2^\#$.

\medskip It is possible to overcome the limitation $p\le2^\#$ by considering the nonlinear diffusion
\be{FDE-sphere}
\frac{\partial w}{\partial t}=\Delta w^m\,,\quad w(t=0,\cdot)=f^p\,,
\ee
for some appropriately chosen $m$. The idea was introduced in~\cite{Demange_2008} and systematized in~\cite{Dolbeault_2014,DEKL,1504}. Let us give some details on how to use~\eqref{FDE-sphere}. We assume that
\be{Range:p-NL}
p\in[1,2)\cup(2,2^*]\quad\mbox{if}\quad d\ge3\quad\mbox{and}\quad p\in[1,2)\cup(2,+\infty)\quad\mbox{if}\quad d=1,\,2\,.
\ee
a range which covers the case \hbox{$2^\#< p<2^*$} and also $p=2^*$ if $d\ge3$. It is convenient to replace the solution~$w$ of~\eqref{FDE-sphere} by $u$ such that $w=u^{\beta\,p}$ with
\[
\beta=\frac2{2-p\,(1-m)}
\]
so that $u$ solves the nonlinear diffusion equation
\be{NLeqn}
\frac{\partial u}{\partial t}=u^{2-2\beta}\(\Delta u+\kappa\,\frac{|\nabla u|^2}u\)\,,\quad u(t=0,\cdot)=|f|^{1/\beta}\,,
\ee
with $\kappa=\beta\,(p-2)+1$. Then $h=w^{1/p}=u^\beta$ is such that
\begin{multline*}
\frac d{dt}\left[\nrmSd{\nabla h}2^2-\frac d{p-2}\(\nrmSd hp^2-\nrmSd h2^2\)\right]\\
=-\,2\,\beta^2\iSd{\(\tfrac d{d-1}\,|Q_u|^2+\delta(\beta)\,\frac{|\nabla u|^4}{u^2}\)}
\end{multline*}
with $Q_u:=\mathrm{Hess}u-\frac1d\,\Delta u\,\mathrm{Id}-\beta\,(p-1)\,\frac{d-1}{d+2}\(\frac{\nabla u\otimes\nabla u}u-\frac1d\,\frac{|\nabla u|^2}u\,\mathrm{Id}\)$ and
\be{gamma}
\delta(\beta):=-\(\frac{d-1}{d+2}\,(\kappa+\beta-1)\)^2+\,\kappa\,(\beta-1)+\,\frac d{d+2}\,(\kappa+\beta-1)\,.
\ee
As a consequence, we can write an improved version of~\eqref{Ineq:GNS}, with an integral remainder term
\begin{multline*}
\nrmSd{\nabla f}2^2-\frac d{p-2}\(\nrmSd fp^2-\nrmSd f2^2\)\\
\ge2\,\beta^2\int_0^{+\infty}\iSd{\(\tfrac d{d-1}\,|Q_u|^2+\delta(\beta)\,\frac{|\nabla u|^4}{u^2}\)}\,dt\,.
\end{multline*}
We recall that the optimal constant cannot be larger than $d$, as shown by a Taylor expansion around $f=1$ of~\eqref{Ineq:GNS} and by the Poincar\'e inequality on $\mathbb S^d$ with optimal constant, which corresponds to the case $p=1$ in~\eqref{Ineq:GNS}. This explains why, along the flow~\eqref{FDE-sphere}, optimality in~\eqref{Ineq:GNS} is achieved only in the limit as $t\to+\infty$. The infimum of the quotient
\[
\mathcal Q_p[h]:=\frac{(2-p)\iSd{|\nabla h|^2}}{\iSd{|h|^2}-\(\iSd{|h|^p}\)^{2/p}}
\]
is obtained by considering $\lim_{\varepsilon\to0}\mathcal Q_p[1+\varepsilon\,h]=\mathcal Q_1[h]$.

Also notice that $\iSd{\(\tfrac d{d-1}\,|Q_u|^2+\delta(\beta)\,\frac{|\nabla u|^4}{u^2}\)}$ written for $u=1+\varepsilon\,h$ with $h$ of zero average gives, at order $\varepsilon^2$, the quadratic form
\[
h\mapsto\frac d{d-1}\iSd{\left|\mathrm{Hess}h-\frac1d\,\Delta h\,\mathrm{Id}\right|^2}=\iSd{(\Delta h)^2}
\]
and that, as a consequence of an expansion of the square in $\iSd{(\Delta h+d\,h)^2}$ and of the Poincar\'e inequality on $\mathbb S^d$,
\[
\iSd{(\Delta h)^2}\ge d\iSd{|\nabla h|^2}\,.
\]
The computation is therefore optimal for the solutions of~\eqref{NLeqn} in the asymptotic regime as $t\to+\infty$. In the nonlinear regime, the constant which appears in the \emph{carr\'e du champ} computation is also $d$, which guarantees that the method gives the optimal constant in~\eqref{Ineq:GNS}. In the words of Section~\ref{Sec:LinearFlows} and with a straightforward analogy with Proposition~\ref{Prop:New1}, we have $\kappa_2=\kappa_1=\kappa_0$. This feature will again appear in Sections~\ref{Sec:FDE} and~\ref{Sec:CKN}. With the \emph{carr\'e du champ} method and the flow~\eqref{NLeqn}, not only the results of~\cite{MR1134481} are recovered in a natural parabolic setting, but improvements like~\eqref{improvedineq} and stability results, are also obtained: see~\cite{1504,Dolbeault_2020} for more considerations in this direction.

%%%%%%%%%%%%%%%%%%%%%%%%%%%%%%%%%%%%%%%%%%%%%%%%%%%%%%%%%%%%%%%%%%%%%%%%
%%%%%%%%%%%%%%%%%%%%%%%%%%%%%%%%%%%%%%%%%%%%%%%%%%%%%%%%%%%%%%%%%%%%%%%%
\section{Entropy methods and the fast diffusion equation on the Euclidean space}\label{Sec:FDE}

Most of this section is a summary of the results of~\cite{BDNS2021}. The optimal constant in the entropy-entropy production inequality is given by the spectral gap of the linearized functional inequality. An important difference with Section~\ref{Sec:SphereFlows} is that the linearization has to be done around the Barenblatt profile and not around constant functions. However, from the heuristic point of view, the situation is similar and the strategy is, \emph{mutatis mutandis}, to prove that $\kappa_2=\kappa_0$. This explains why optimal constants are obtained by generalized \emph{carr\'e du champ} methods and why stability results can be achieved for~\eqref{GNS}, which is the major result of~\cite{BDNS2021}.

%%%%%%%%%%%%%%%%%%%%%%%%%%%%%%%%%%%%%%%%%%%%%%%%%%%%%%%%%%%%%%%%%%%%%%%%
\subsection{R\'enyi entropy powers}\label{FDE:Renyi}

On $\R^d$, $d\ge1$, let us consider the \emph{fast diffusion} equation
\be{FD}
\frac{\partial u}{\partial t}=\Delta u^m
\ee
with $m\in[m_1,1)$, $m_1:=(d-1)/d$, and an initial datum $u(t=0,x)=u_0(x)\ge0$, with $u_0\in\mathrm L^1\big(\mathbb R^d,(1+|x|^2)\,dx\big)$. Equation~\eqref{FD} was interpreted as the gradient flow of the \emph{entropy}
\[
\mathsf E:=\ird{u^m}
\]
with respect to Wasserstein's distance in~\cite{Otto2001}. Equation~\eqref{FD} admits self-similar solutions, the so-called \emph{Barenblatt functions} which, up to a scaling, a multiplication by a constant and a translation, take the form
\[
B(t,x)=\frac1{\kappa\,R(t)^d}\,\mathcal B\!\(\tau(t),\frac x{\kappa\,R(t)}\)\quad\forall\,(t,x)\in\R^+\times\R^d
\]
for some numerical constant $\kappa>0$, an unbounded increasing function $t\mapsto R(t)$, $\tau(t)=\frac12\,\log R(t)$ and a Barenblatt profile $\mathcal B$ that can be written as
\[
\mathcal B(x)=\(1+|x|^2\)^\frac1{m-1}\quad\forall\,x\in\R^d\,.
\]

As in Section~\ref{Sec:SphereFlows}, we prove~\eqref{GNS} directly by the \emph{carr\'e du champ method}. The key property, inspired by the \emph{R\'enyi entropy powers} of~\cite{MR3200617} is based~on:\\
(i) The entropy growth estimate:
\[
\mathsf E'=(1-m)\,\mathsf I\quad\mbox{where}\quad\mathsf I:=\ird{u\,\left|\nabla \mathsf P\right|^2}\quad\mbox{and}\quad\mathsf P:=\frac m{m-1}\,u^{m-1}\,.
\]
Here $\mathsf I$ and $\mathsf P$ are known respectively as the generalized \emph{Fisher information} and the \emph{pressure variable}.\\
(ii) The $t$-derivative of the generalized \emph{R\'enyi entropy powers}, that is,
\[
\mathsf G:=\mathsf I\,\mathsf E^{\,2\,\frac{m-m_1}{1-m}}\,,
\]
satisfies the identity
\be{F}
-\frac12\,\frac d{dt}\log\mathsf G=\ird{u^m\,\left|\,\mathrm D^2\mathsf P-\frac 1d\,\Delta\mathsf P\,\mathrm{Id}\,\right|^2}
+(m-m_1)\ird{u^m\,\left|\,\Delta\mathsf P+\frac{\mathsf I}{\mathsf E}\,\right|^2}\,.
\ee
Hence $\mathsf G$ is monotone with a limit given by a self-similar \emph{Barenblatt function} $B(t,x)$. It turns out that writing $\mathsf G[u]\ge\mathsf G[B]$ is exactly~\eqref{GNS} written with the optimal constant, using the relations $f=u^{m-1/2}$, so that, with
\[
p=\frac1{2\,m-1}\,,
\]
we have $\int_{\mathbb R^d}u\,dx=\int_{\mathbb R^d}f^{2p}\,dx$, $\mathsf E=\int_{\mathbb R^d}f^{p+1}\,dx$ and $\mathsf I=(p+1)^2\int_{\mathbb R^d}|\nabla f|^2\,dx$. Inequality~\eqref{GNS} is then easily recovered.

This gives the growth rate of $\mathsf E$. Let
\[
C_0:=4\,\tfrac{(1-m)^3}{(2\,m-1)^2}\,\(\mathcal C_{\mathrm{GNS}}(p)\)^{\frac2d\frac{(d+2)\,m-d}{(1-m)\,(2\,m-1)}}\,M^\frac{(d+2)\,m-d}{d\,(1-m)}
\]
 where $M=\nrm{u_0}1$, and $m_c:=(d-2)/d$. According to~\cite[Lemma~2.1]{BDNS2021}, we have the following result.
%-----------------------------------------------------------------------
\begin{lem}[\cite{BDNS2021}]\label{Lem:Gronwall} Assume that $d\ge1$, $m\in[m_1,1)$ with additionally $m>1/2$ if $d=1$ or $d=2$, and consider a solution of~\eqref{FD} with initial datum $u_0\in\mathrm L^1_+\!\(\R^d,(1+|x|^2)\,dx\)$ such that $u_0^m\in\mathrm L^1(\R^d)$. Then
\be{Ch2:GrowthEntropy}
\ird{u^m(t,x)}\ge\(\ird{u_0^m}+\tfrac{(1-m)\,C_0}{m-m_c}\,t\)^\frac{1-m}{m-m_c}\quad\forall\,t\ge0\,.
\ee
\end{lem}
%-----------------------------------------------------------------------
Knowing the optimal constant $C_0$ in~\eqref{Ch2:GrowthEntropy} is equivalent to identifying the optimal constant in~\eqref{GNS}. This is a standard property of entropy methods that optimal decay rates in the evolution equation are equivalent to optimal constants in the corresponding functional inequalities. The threshold case $m=m_1$ in~\eqref{FD} corresponds to $p=p_\star$ in~\eqref{GNS}, \emph{i.e.}, to the Sobolev inequality.

%%%%%%%%%%%%%%%%%%%%%%%%%%%%%%%%%%%%%%%%%%%%%%%%%%%%%%%%%%%%%%%%%%%%%%%%
\subsection{Relative entropy and relative Fisher information}\label{Sec:Self-Similar}

Proving~\eqref{F} requires some integrations by parts which are delicate to justify in the class of solutions considered in Lemma~\ref{Lem:Gronwall}. A rigorous proof goes through the time-dependent rescaling
\be{TDRs}
u(t,x)=\frac1{\kappa^d\,R^d}\,v\!\(\tau,\frac x{\kappa\,R}\)\quad\mbox{where}\quad\frac{dR}{dt}=R^{-d\,(m-m_c)}\,,\quad\tau(t):=\tfrac12\,\log R(t)\,,
\ee
so that~\eqref{FD} is changed, with the choice $R(0)=1$, into the \emph{Fokker-Planck type equation}
\be{RFD}
\frac{\partial v}{\partial\tau}+\nabla\cdot\Big[v\(\nabla v^{m-1}-\,2\,x\)\Big]=0\,,
\ee
with an initial datum $v_0:=\kappa^d\,u_0(\kappa\cdot)\ge0$, for some numerical parameter~$\kappa$ which depends only on $m$. For simplicity, we shall assume from now on that $\ird{v_0}=\ird{\mathcal B}$.

For a given function $v\in\mathrm L^1_+\!\(\R^d,(1+|x|^2)\,dx\)$ such that $v^m\in\mathrm L^1(\R^d)$ and $\ird v=\ird{\mathcal B}$, let us define the \emph{relative entropy} (or \emph{free energy}) and the \emph{relative Fisher information} respectively by 
\[
\mathcal F[v]\!:=\!\int_{\mathbb R^d}\(\mathcal B^{m-1}\,(v-\mathcal B)-\tfrac1m\,(v^m-\mathcal B^m)\)dx\,,\quad\mathcal I[v]\!:=\!\int_{\mathbb R^d}v\left|\nabla v^{m-1}+\,2\,x\right|^2\,dx\,.
\]
It is proved in~\cite{DelPino2002} that the \emph{entropy-entropy production inequality}
\be{EEPineq}
\mathcal I[v]\ge4\,\mathcal F[v]
\ee
is exactly equivalent to~\eqref{GNS} written with the optimal constant. This is not difficult to check using $f=v^{m-1/2}$ and $\mathcal B=\mathsf g^{2p}$ with $p=1/(2\,m-1)$.

If additionally $v$ solves~\eqref{RFD}, it is almost straightforward to check that
\be{EP1}
\frac d{dt}\mathcal F[v(t,\cdot)]=-\,\mathcal I[v(t,\cdot)]\le-\,4\,\mathcal F[v(t,\cdot)]\,,
\ee
so that we have the estimate
\[
\mathcal F[v(t,\cdot)]\le\mathcal F[v_0]\,e^{-4t}\quad\forall\,t\ge0\,.
\]
This can be used to obtain intermediate asymptotics for a solution of~\eqref{FD} using~\eqref{TDRs} and a generalized Pinsker-Csisz\'ar-Kullback inequality. Even more interesting is the fact that the \emph{carr\'e du champ} method applies:
\be{EP2}
\frac d{dt}\mathcal I[v(t,\cdot)]\le-\,4\,\mathcal I[v(t,\cdot)]\,,
\ee
with equality up to a term which is similar to the one in~\eqref{F} after applying~\eqref{TDRs}. The justification of the integrations by parts is easier than in the R\'enyi entropy powers approach, but still requires some care: see~\cite{MR3497125,DEL-JEPE} and references therein for a complete proof. This is also the right setting to prove the identities in the R\'enyi entropy framework of Section~\ref{FDE:Renyi}, up to the change of variables~\eqref{TDRs}: all details can be found in~\cite{DEL-JEPE}.

Let us consider the quotient
\[
\mathcal Q[v]:=\frac{\mathcal I[v]}{\mathcal F[v]}\,.
\]
Based on~\eqref{EP1} and~\eqref{EP2}, we obtain that
\[
\frac d{dt}\mathcal Q[v(t,\cdot)]=\mathcal Q[v(t,\cdot)]\,\big(\mathcal Q[v(t,\cdot)]-4\big)
\]
if $v$ solves~\eqref{RFD}. This proves the following \emph{initial time layer} property.
%-----------------------------------------------------------------------
\begin{lem}[\cite{BDNS2021}]\label{Lem:ITlayer} With the above assumptions, if $\mathcal Q[v(T,\cdot)]\ge4+\eta$ for some $\eta>0$ and $T>0$, then $\mathcal Q[v(t,\cdot)]\ge\,4+4\,\eta\,e^{-4\,T}/(4+\eta-\eta\,e^{-4\,T})$ for any $t\in[0,T]$.\end{lem}
%-----------------------------------------------------------------------

%%%%%%%%%%%%%%%%%%%%%%%%%%%%%%%%%%%%%%%%%%%%%%%%%%%%%%%%%%%%%%%%%%%%%%%%
\subsection{Linearization and Hardy-Poincar\'e inequalities}

For reasons which are similar to the ones in Section~\ref{Sec:SphereFlows}, it turns out that $\inf\mathcal Q[v]=4$ is not achieved among admissible functions, in the sense that $\lim_{n\to+\infty}\mathcal F[v_n]=0$ if $(v_n)_{n\in\N}$ is a minimizing sequence for $\mathcal Q$. In other words and under appropriate normalization conditions, this means that $v_n\to\mathcal B$ and suggests to consider the limiting problem
\[
\inf\frac{\mathsf I[h]}{\mathsf F[h]}=\lim_{\varepsilon\to0}\mathcal Q\big[\mathcal B\,\big(1+\varepsilon\,\mathcal B^{1-m}\,h\big)\big]
\]
where the \emph{linearized free energy} and the \emph{linearized Fisher information} are defined respectively by
\[\label{linearized.fisher}
\mathsf F[h]:=\frac m2\ird{|h|^2\,\mB^{2-m}}\quad\mbox{and}\quad\mathsf I[h]:=m\,(1-m)\ird{|\nabla h|^2\,\mB}\,.
\]
By the \emph{Hardy-Poincar\'e inequality} if $d\ge1$ and $m\in(m_1,1)$, and for any function $h\in\mathrm L^2(\R^d,\mB^{2-m}\,dx)$ such that $\nabla h\in\mathrm L^2(\R^d,\mB\,dx)$ and $\ird{h\,\mB^{2-m}}=0$, we have the \emph{Hardy-Poincar\'e inequality}
\[
\mathsf I[h]\ge4\,\mathsf F[h]\,.
\]
If additionally we assume that $\ird{x\,h\,\mB^{2-m}}=0$, then we have the \emph{improved Hardy-Poincar\'e inequality}
\be{HP-PNAS}
\mathsf I[h]\ge4\,\alpha\,\mathsf F[h]
\ee
where $\alpha=2-d\,(1-m)$ if $m\in(m_1,1)$. To get an improved inequality if $m=m_1$, one has to assume additionally that $\ird{|x|^2\,h\,\mB^{2-m}}=0$. As in~\cite{Blanchet2009,Bonforte2010c,BDNS2021}, the improved estimate of~\eqref{HP-PNAS} can be reimported in the nonlinear functionals as follows.
%---------------------------------------------------------------------
\begin{lem}[\cite{BDNS2021}]\label{Prop:Gap} Let $m\in(m_1,1)$ if $d\ge2$, $m\in(1/3,1)$ if $d=1$, $\eta=2\,d\,(m-m_1)$. There exists an explicit $\chi\in(0,1)$ such that,if $\ird v=\Mstar$, $\ird{x\,v}=0$ and
\be{BarenblattStable}
(1-\varepsilon)\,\mB\le v\le(1+\varepsilon)\,\mB
\ee
for some $\varepsilon\in(0,\chi\,\eta)$, then
\[
\mathcal Q[v]\ge4+\eta\,.
\]
\end{lem}
%---------------------------------------------------------------------
In other words, since Condition~\eqref{BarenblattStable} is stable under the action of~\eqref{RFD} according to~\cite{Blanchet2009}, we have the \emph{asymptotic time layer} property $\mathcal Q[v(t,\cdot)]\ge4+\eta$ if the solution $v(t,\cdot)$ of~\eqref{RFD} satisfies~\eqref{BarenblattStable} for any $t>T_\star$, for some \emph{threshold time} $T_\star>0$. 

%%%%%%%%%%%%%%%%%%%%%%%%%%%%%%%%%%%%%%%%%%%%%%%%%%%%%%%%%%%%%%%%%%%%%%%%
\subsection{The threshold time estimate}

The existence of a \emph{threshold time} is known for a given initial datum from~\cite{Bonforte2006}, but has been quantified only in~\cite{BDNS2021} as follows.
%-----------------------------------------------------------------------
\begin{thm}[\cite{BDNS2021}]\label{Prop:TT} Let $m\in\big[m_1,1\big)$ if $d\ge2$, $m>1/3$ if $d=1$, $A>0$, and $G>0$ be given. Let $\varepsilon\in(0,\min\{\chi\,\eta, \varepsilon_0\})$, with $\eta$ and $\chi$ as in Proposition~\ref{Prop:Gap}, for some explicit $\varepsilon_0\in (0,1/2)$, and define
\be{Tstar}
T_\star:= \frac1{2\,\alpha}\,\log\(1+\alpha\,\taustar\,\frac{1+A^{1-m}+G^\frac\alpha2}{\varepsilon^\mathsf a}\)
\ee
with $\alpha=d\,(m-m_c)$, for some explicit, positive numerical constants $\taustar>0$ and~$\mathsf a$. Then for any solution $v$ to~\eqref{RFD} with nonnegative initial datum $v_0\in\mathrm L^1(\R^d)$, $\ird{v_0}=\Mstar$, \hbox{$\ird{x\,v_0}=0$} which satisfies
\be{hyp:Harnack.self}
\sup_{r>0}r^\frac{d\,(m-m_c)}{(1-m)}\int_{|x|>r}u_0\,dx\le A\,,\quad \mathcal F[v_0]\le G\,,
\ee
we have that
\be{uniformFDr}
(1-\varepsilon)\,\mB(x)\le v(t,x)\le(1+\varepsilon)\,\mB(x)\quad\forall\,(t,x)\in[T_\star,+\infty)\times\R^d\,.
\ee
\end{thm}
%-----------------------------------------------------------------------
Based on a \emph{global Harnack Principle}, this result is at the core of the method of~\cite{BDNS2021}. It relies on a quantitative version of the results of J.~Moser in~\cite{Moser1964,Moser1971}, and a constructive proof of~\cite{Bonforte2006} based on the improved results of~\cite{Bonforte2019a}. As discussed in~\cite[Chapter~7]{BDNS2021}, the tail decay in~\eqref{hyp:Harnack.self} is not only sufficient but also necessary for obtaining~\eqref{uniformFDr}.

%%%%%%%%%%%%%%%%%%%%%%%%%%%%%%%%%%%%%%%%%%%%%%%%%%%%%%%%%%%%%%%%%%%%%%%%
\subsection{A stability result for Gagliardo-Nirenberg-Sobolev inequalities}\label{Sec:GNS-stability}

By collecting the \emph{initial time layer} property of of Lemma~\ref{Lem:ITlayer}, the \emph{asympotic time layer} property of Lemma~\ref{Prop:Gap} and the \emph{ threshold time estimate} of Theorem~\ref{Prop:TT}, we are able to write an improved entropy-entropy production inequality which amounts to $\mathcal Q[v(t,\cdot)]\ge\,4+4\,\eta\,e^{-4\,T_\star}/(4+\eta-\eta\,e^{-4\,T_\star})$ for any $t\in\R^+$. This is the desired stability result, which goes as follows. There is a unique $\mathsf g_f\in\mathfrak M$ such that
\[
\ird{(1,x,|x|^2)\,f^{2p}}=\ird{(1,x,|x|^2)\,\mathsf g_f^{2p}}\,.
\]
By the entropy-entropy production inequality~\eqref{EEPineq} written for $v=f^{2p}$, we know that
\[
\delta[f]:=\frac{p+1}{p-1}\ird{\left|(p-1)\,\nabla f+f^p\,\nabla \mathsf g_f^{1-p}\right|^2}-\,4\,\mathcal E[f]
\]
is nonnegative, where the \emph{relative entropy} with respect to $\mathsf g_f$ is defined by
\[
\mathcal E[f]:=\frac{2\,p}{1-p}\ird{\(f^{p+1}-\mathsf g_f^{p+1}-\tfrac{1+p}{2\,p}\,\mathsf g_f^{1-p}\(f^{2p}-\mathsf g_f^{2p}\)\)}\,.
\]
An expansion of $\delta$ shows that, for some explicit numerical constant $\mathcal K_{\mathrm{GNS}}$,
\[
\delta[f]=(p-1)^2\,\nrm{\nabla f}2^2+4\,\frac{d-p\,(d-2)}{p-1}\,\nrm f{p+1}^{p+1}-\mathcal K_{\mathrm{GNS}}\,\frac{p+1}{p-1}\,\nrm f{2p}^{2p\gamma}
\]
and the inequality $\delta[f]\ge0$ is actually equivalent to~\eqref{GNS}.
%-----------------------------------------------------------------------
\begin{thm}[Stability in Gagliardo-Nirenberg inequalities, \cite{BDNS2021}]\label{Thm:stabilityDraft2} Let $d\ge1$ and $p\in(1,p^\star)$. There is an explicit $\mathcal C$ such that, for any nonnegative $f\in\mathrm L^{2p}\big(\mathbb R^d,(1+|x|^2)\,dx\big)$ such that $\nabla f\in\mathrm L^2(\mathbb R^d)$ and
\[
A=\sup_{r>0}r^\frac{d-p\,(d-4)}{p-1}\int_{|x|>r}f^{2p}\,dx<\infty\quad\mbox{and}\quad\mathcal E[f]=G<\infty\,,
\]
then we have
\[
\delta[f]\ge\mathcal C\inf_{\varphi\in\mathfrak M}\ird{\left|(p-1)\,\nabla f+f^p\,\nabla \varphi^{1-p}\right|^2}\,.
\]
\end{thm}
%-----------------------------------------------------------------------
The main point is that $\mathcal C$ has an explicit expression in terms of $A$ and $G$, which does not degenerate to  zero if $f\in\mathfrak M$. The critical case \hbox{$p=p^\star$} can also be covered up to an additional scaling: see~\cite[Chapter~6]{BDNS2021}.
%---------------------------------------------------------------------
\begin{thm}[Stability in Sobolev inequalities, \cite{BDNS2021}]\label{Thm:Main} Let $d\ge3$ and $A>0$.There is an explicit $\mathcal C$ such that, for any nonnegative function $f\in\mathrm L^{2^*}\big(\mathbb R^d,(1+|x|^2)\,dx\big)$ such that $\nabla f\in\mathrm L^2(\mathbb R^d)$ and
\[
\ird{(1,x, |x|^2)\,f^{2^*}}=\ird{(1,x,|x|^2)\,\mathsf g}\quad\mbox{and}\quad
\sup_{r>0}r^d\int_{|x|>r}\,f^{2^*}\,dx\le A\,,
\]
we have
\be{stability-fisher}
\nrm{\nabla f}2^2-\mathsf S_d\,\nrm f{2^*}^2\ge\mathcal C\ird{\left|\tfrac2{d-2}\,\nabla f+f^\frac d{d-2}\,\nabla\mathsf g^{-\frac2{d-2}}\right|^2}\,.
\ee
\end{thm}
%---------------------------------------------------------------------
The stability constant is $\mathcal C$ in~Theorem~\ref{Thm:Main} depends only on~$A$. In Theorems~\ref{Thm:stabilityDraft2} and~\ref{Thm:Main}, the stability is measured by the \emph{relative Fisher information}, that is,~by
\[
\ird{\left|(p-1)\,\nabla f+f^p\,\nabla \varphi^{1-p}\right|^2}\,.
\]
Such a quantity differs from the distance to $\mathfrak M$ in~\eqref{Bianchi-Egnell} and deserves some comments. A simple expansion of the square and an optimization under scalings proves the following nonlinear extension of the \emph{Heisenberg uncertainty principle},
\be{Heisenberg}
\(\frac d{p+1}\ird{f^{p+1}}\)^2\le\ird{|\nabla f|^2}\ird{|x|^2\,f^{2p}}\,.
\ee
As a consequence of~\eqref{GNS}, \emph{i.e.}, of $\delta[f]\ge0$, we have that 
\[
\ird{\left|(p-1)\,\nabla f+f^p\,\nabla \varphi^{1-p}\right|^2}\ge4\,\frac{p-1}{p+1}\,\mathcal E[f]\,.
\]
According to~\cite[Lemma~1.7]{BDNS2021}, by the \emph{Pinsker-Csisz\'ar-Kullback inequality}, $\mathcal E[f]$ controls a more standard distance between $f^{2p}$ and $\mathsf g_f^{2p}$, namely
\[
\mathcal E[f]\ge\frac{p+1}{8\,p}\(\ird{\mathsf g_f^{3p-1}}\)^{-1}\,\left|f^{2p}-\mathsf g_f^{2p}\right|_1^2\,.
\]

%%%%%%%%%%%%%%%%%%%%%%%%%%%%%%%%%%%%%%%%%%%%%%%%%%%%%%%%%%%%%%%%%%%%%%%%
%%%%%%%%%%%%%%%%%%%%%%%%%%%%%%%%%%%%%%%%%%%%%%%%%%%%%%%%%%%%%%%%%%%%%%%%
\section{Symmetry and symmetry breaking in Caffarelli-Kohn-Nirenberg inequalities}\label{Sec:CKN}

This section is devoted to Theorem~\ref{Thm:CKN}. The interpretation of~\eqref{CKN} as an entropy-entropy production inequality paves the way to a flow method which generalizes the computations done for the fast diffusion equation in Section~\ref{Sec:FDE}. By the \emph{carr\'e du champ} method, we are able to reduce the symmetry issue to a simple dichotomy:\\
$\rhd$ either the entropy-entropy production inequality is linearly unstable around the Barenblatt type profiles and optimal functions are not radially symmetric,\\
$\rhd$ or the linearized entropy-entropy production inequality admits a spectral gap which also enters in the \emph{carr\'e du champ} estimates, in the nonlinear regime, as for the fast diffusion equation, and then the inequality holds with precisely the spectral gap as optimal constant. Using the analogy with Section~\ref{Sec:SphereFlows}, we are again in a case for which $\kappa_2=\kappa_0$. More details can be found in~\cite[Section~4.2]{DEL-JEPE}.

%%%%%%%%%%%%%%%%%%%%%%%%%%%%%%%%%%%%%%%%%%%%%%%%%%%%%%%%%%%%%%%%%%%%%%%%
\subsection{Symmetry versus symmetry breaking}

Let us consider the \emph{critical Caffarelli-Kohn-Nirenberg inequalities}~\eqref{CKN}. In dimension $d\ge2$, it has been proved in~\cite{Catrina2001} that an optimal function $f\ge0$ exists, if $b<a+1$, and also $b>a$ if $a<0$. With the appropriate normalization, $f$ solves
\be{equation}
-\mathrm{div}\big(\kern 0.5pt|x|^{-2a}\,\nabla f\big) = |x|^{-\kern 0.5pt b\kern 0.5pt p}\,f^{p-1} \,.
\ee
The \emph{symmetry versus symmetry breaking issue} is to decide whether or not $f$ is radially symmetric. The answer is given by Theorem~\ref{Thm:CKN}, but a slightly more general \emph{rigidity result} is given in~\cite{Dolbeault2016}: if $ b\le b_{\,\rm FS}(a)<b<a+1$, up to a scaling, $\mathsf g$ as defined in~\eqref{g-CKN} is the unique positive solution of~\eqref{equation}. This symmetry result cannot be handled so far with symmetrization methods: in the range of $(a,b)$~for~which
\[
u\mapsto\ird{\frac{|\nabla u|^2}{|x|^{2\,a}}}-\frac{\nrm{\nabla\mathsf g\,|x|^{-a}}2^2}{\nrm{\mathsf g\,|x|^{-b}}p^2}\(\ird{\frac{|u|^p}{|x|^{b\,p}}}\)^{2/p}
\]
is linearly stable around $\mathsf g$, a local property, then $\mathsf g$ is a global minimizer and, as a consequence, $\mathsf C_{a,b}=\nrm{\nabla\mathsf g\,|x|^{-a}}2^2/\nrm{\mathsf g\,|x|^{-b}}p^2$.

%%%%%%%%%%%%%%%%%%%%%%%%%%%%%%%%%%%%%%%%%%%%%%%%%%%%%%%%%%%%%%%%%%%%%%%%
\subsection{Strategy of the proof}\label{Sec:CKNstrategy}

We summarize here the strategy of the proof in~\cite{Dolbeault2016}, which we decompose in four steps.

%.......................................................................
\subsubsection{A modified Sobolev inequality}\label{Sec:CKN-S}

Let us use spherical coordinates with $r=|x|$ and $\omega=x/r$ for any $x\in\R^d\setminus\{0\}$. With the change of variables
\[
f(r,\omega)=F(s,\omega)\quad\mbox{with}\quad s=r^\alpha\,,
\]
Inequality~\eqref{CKN} is transformed into a Sobolev type inequality
\be{sob}
\irdmu{|\D F|^2}\ge\mathsf C^{\kern 0.25pt d}_{a,b}\,\alpha^{1-\frac2p}\(\irdmu{|F|^p}\)^\frac2p
\ee
where $d\mu=s^{n-1}\,ds\,d\omega=|x|^{n-d}\,dx$,
\[
\D F:=\(\alpha\,\frac{\partial F}{\partial s},\frac1s\,\nabla_\omega F\)
\]
so that $|\D F|^2=\alpha^2\,\big(\frac{\partial F}{\partial s}\big)^2+\frac{|\nabla_\omega F|^2}{s^2}$, and $\alpha$ and $n$ are two parameters defined by
\[
\alpha=\frac{(1+a-b)\,(a_c-a)}{a_c-a+b}\,,\quad n=\frac{2\,p}{p-2}\,.
\]
Here $n$ plays the role of an artificial dimension while $\alpha$ reflects a discrepancy between derivatives with respect to the radial and angular variables. The parametrization by $a$ and $b$ of~\eqref{CKN} is now replaced by the parametrization by $\alpha>0$ and $n>d$. The Felli \& Schneider curve $b=b_{\kern 0.5pt\rm FS}(a)$ becomes $\alpha=\alpha_{\rm FS}$ in the new set of parameters $\alpha$ and $n$, with
\[
\alpha_{\rm FS}:=\sqrt{\frac{d-1}{n-1}}\,.
\]
For $\alpha>\alpha_{\rm FS}$ the minimizers are not radial. With $u=|F|^p=\P^{-n}$, Inequality~\eqref{sob} becomes
\be{sob2}
\irdmu{u\,|\D \P|^2}\ge\frac{4\,\alpha^{1-\frac2p}}{(n-2)^2}\,\mathsf C^{\kern 0.25pt d}_{a,b}\(\irdmu u\)^\frac2p
\ee
with $p=2\,n/(n-2)$ while~\eqref{equation} is transformed into
\be{sobequn}
-\,\mathcal L\,F=F^{p-1}\quad\mbox{where}\quad\mathcal L:=-\,\Dstar\cdot\D\,.
\ee

%.......................................................................
\subsubsection{The flow}\label{Sec:Flow}

The fast diffusion flow
\be{flow}
\frac{\partial\kern 0.5pt u}{\partial t}=\mathcal L\,u^{1-\frac1n}=(1-n)\,\Dstar\cdot\big(u\,\D\P\big)
\ee
is similar to the flow studied in Section~\ref{FDE:Renyi} and admits self-similar solutions
\[
B(t,x)=t^{-n}\(c+\frac{|x|^2}{2\,(n-1)\,\alpha^2\,t^2}\)^{-n}
\]
where $c>0$ can be adjusted so that $\irdmu{B(t,x)}=$ for any given $M>0$. Let us notice that $M=\irdmu u$ is conserved by~\eqref{flow} and formally compute
\[
\frac d{dt}\irdmu{u\,|\D\P|^2}=-\,2\irdmu{\(\frac12\,\mathcal L\,|\D\P|^2-\D\P\cdot\D\mathcal L\,\P-\frac1n\,(\mathcal L\,\P)^2\)}\,.
\]
Hence, if $u$ is optimal for~\eqref{sob2}, it is clear that $\frac d{dt}\irdmu{u\,|\D\P|^2}=0$.

%.......................................................................
\subsubsection{The \emph{carr\'e du champ} method}\label{Sec:CCM}

Let us denote by $\nabla$ the gradient with respect to the angular variable $\omega\in\S^{d-1}$ and by $\Delta$ the Laplace-Beltrami operator on $\S^{d-1}$. The following algebraic computation is adapted from~\cite{Dolbeault2016} and extends to~\eqref{flow} the \emph{carr\'e du champ} method. With $n>d\ge3$ and the notation $'=\partial_s$, we have
\begin{multline*}
\frac12\,\mathcal L\,|\D\P|^2-\D\P\cdot\D\mathcal L\,\P-\frac1n\,(\mathcal L\,\P)^2\\
=\alpha^4\,\frac{n-1}n\(\P''-\frac{\P'}r-\frac{\Delta\kern 0.5pt\P}{\alpha^2\,(n-1)\,r^2}\)^2
+\frac{2\,\alpha^2}{r^2}\Big|\nabla\P'-\frac{\nabla\P}r\Big|^2\\
+\,\frac1{r^4}\(\frac12\,\Delta\kern 0.5pt|\nabla\P|^2-\nabla\P\cdot\nabla\Delta\kern 0.5pt\P-\frac1{n-1}(\Delta\kern 0.5pt\P)^2-(n-2)\,\alpha^2\,|\nabla\P|^2\)\,.
\end{multline*}
If $\P$ is a positive function in $C^3(\mathbb S^{d-1})$ and $d\ge3$, the computation of Section~\ref{Sec:SphereFlows} can be adapted to prove that
\begin{eqnarray*}
&&\kern-24pt\int_{\mathbb S^{d-1}} \(\frac12 \,\Delta\kern 0.5pt|\nabla \P|^2 - \nabla \P \cdot \nabla \Delta\kern 0.5pt \P - \frac1{n-1}(\Delta\kern 0.5pt \P)^2 - (n-2)\,\alpha^2\,|\nabla \P|^2\) \P^{1-n}\,d\omega\\
&=&\frac{(n-2)\,(d-1)}{(n-1)\,(d-2)}\int_{\mathbb S^{d-1}}|\mathsf Q|^2\,\P^{1-n}\,d\omega\\
&&+\,\frac{n-d}{2\,(d+1)} \(\frac{n+3}2 + \frac{3\,(n-1)\,(n+1)\,(d-2)}{2\,(n-2)\,(d+1)} \) \int_{\mathbb S^{d-1}} { \frac{|\nabla \P|^4}{\P^2} \P^{1-n}}\,d\omega\\
&&+\,(n-2) \(\alpha_{\rm FS}^2 - \alpha^2\) \int_{\mathbb S^{d-1}}{ |\nabla \P|^2\,\P^{1-n}}\,d\omega
\end{eqnarray*}
where $\mathsf Q:=(\nabla\otimes\nabla)\,\P-\frac1{d-1}\,(\Delta\kern 0.5pt\P)\,\mathrm{Id}-\frac{3\,(n-1)\,(n-d)}{2\,(n-2)\,(d+1)}\(\frac{\nabla\P\otimes\nabla\P}\P-\frac1{d-1}\,\frac{|\nabla\P|^2}\P\,\mathrm{Id}\)$. For the case $d=2$ we refer the reader to~\cite{Dolbeault2016}.
%-----------------------------------------------------------------------
\begin{lem} \label{Carre} Assume that $d\ge3$. If $u$ is a smooth enough and sufficiently decaying function as $|x|=r\to+\infty$, then
\begin{eqnarray*}
&&\kern-24pt\irdmu{\(\frac12\,\mathcal L\,|\D\P|^2-\D\P\cdot\D\mathcal L\,\P-\frac1n\,(\mathcal L\,\P)^2\)}\\
&=&\alpha^4\,\frac{n-1}n\irdmu{\(\P''-\frac{\P'}r-\frac{\Delta\kern 0.5pt\P}{\alpha^2\,(n-1)\,r^2}\)^2}
+2\,\alpha^2\irdmu{\frac1{r^2}\Big|\nabla\P'-\frac{\nabla\P}r\Big|^2}\\
&&+\,\frac{(n-2)\,(d-1)}{(n-1)\,(d-2)}\irdmu{\frac1{r^4}\,|\mathsf Q|^2\,\P^{1-n}}\\
&&+\,\frac{n-d}{2\,(d+1)}\(\frac{n+3}2+\frac{3\,(n-1)\,(n+1)\,(d-2)}{2\,(n-2)\,(d+1)}\)\irdmu{\frac1{r^4}\,\frac{|\nabla\P|^4}{\P^2}\P^{1-n}}\\
&&+\,(n-2) \(\alpha_{\rm FS}^2-\alpha^2\)\irdmu{\frac1{r^4}|\nabla\P|^2\,\P^{1-n}}\,.
\end{eqnarray*}
\end{lem}
%-----------------------------------------------------------------------
Up to regularity issues and decay properties as $|x|\to+\infty$, we conclude that a solution of~\eqref{flow} such that $\frac d{dt}\irdmu{u\,|\D\P|^2}=0$ satisfies $\nabla\P=0$ and \hbox{$\P''-\P'/r=0$} if $\alpha\le\alpha_{\rm FS}$, which means that it takes the form $u^{-n}=\P=\mathsf a+\mathsf b\,r^2$ for some $t$-dependent coefficients $\mathsf a$ and~$\mathsf b$. Inserting this expression in~\eqref{flow}, we obtain $u=B$, up to a scaling, that is, a translation in time in self-similar variables.

%.......................................................................
\subsubsection{Regularity issues}\label{Sec:Reg}
The above computation is so far formal, as integrations by parts are not justified. In~\cite{Dolbeault2016}, the regularity and decay properties needed for the method are established for positive solutions of~\eqref{sobequn} by elliptic estimates. In~\cite{DEL-JEPE}, partial results are obtained for the parabolic point of view using self-similar variables as in Section~\ref{Sec:Self-Similar}. A complete parabolic proof is however still missing.

%%%%%%%%%%%%%%%%%%%%%%%%%%%%%%%%%%%%%%%%%%%%%%%%%%%%%%%%%%%%%%%%%%%%%%%%
%%%%%%%%%%%%%%%%%%%%%%%%%%%%%%%%%%%%%%%%%%%%%%%%%%%%%%%%%%%%%%%%%%%%%%%%
\section{Conclusion and open problems}\label{Sec:Conclusion}

Entropy methods provide a framework which is extremely convenient for solving some fundamental variational problems and identifying the optimal constant in the corresponding functional inequalities. In some cases, evolving the functional by a nonlinear diffusion amounts to test the equation by a complicated expression involving a second order differential operator, powers of the solution and eventually homogeneous weights, which gives rises to delicate and complex computations. The use of a flow and of the corresponding entropy functional provides a framework to reinterpret the various terms and order the computations, as in Section~\ref{Sec:CKNstrategy}. Identifying the optimality cases becomes also much easier and the stationarity with respect to the flow even provides uniqueness results which would be otherwise difficult to guess or to reduce to more standard techniques. Parabolic equations come with some extra properties, like regularization effects, which are crucial to obtain the stability results of Section~\ref{Sec:GNS-stability}. However, its major advantage is to relate the nonlinear regime with the linearized one around self-similar or stationary solutions, using quantities which are meaningful in the two cases: from the gradient flow point of view in the nonlinear regime and in terms of orthogonality in the linear regime. It is a remarkable feature that the use of a parabolic flow allows to bypass symmetrization techniques. This opens new directions of research in non-real valued problems (complex valued functions in quantum mechanics, in presence of magnetic fields, see for instance~\cite{Bonheure_2019,Bonheure_2020}, or systems of equations).

Entropy methods have been introduced in partial differential equations, see~\cite{MR1842428}, in order to handle systems of charged particles, with the aim of obtaining rates in kinetic and related equations. There is a whole area of research which has emerged during the last 20 years under the name of \emph{hypocoercivity}, with $\mathrm H^1$ methods (see~\cite{Mem-villani}) or $\mathrm L^2$ methods (see~\cite{DMS-2part,arnold2021sharpening}). In $\mathrm L^2$ hypocoercive approaches, there is a simple strategy: if the kinetic equation admits a diffusion limit (under the appropriate parabolic scaling) whose asymptotic behaviour is governed by a functional inequality, then the corresponding rates of convergence or decay can be reimported in the kinetic equation: see for instance~\cite{BDLS2020,BDS-very-weak,BDMMS}. This is even true for systems with a non-local Poisson coupling, as shown in~\cite{Addala_2021}. However, identifying sharp rates and relating nonlinear models with their linearized counterparts, as can be done in the above examples of linear and nonlinear parabolic equations, is not done yet, even in the simplest benchmark cases considered in~\cite{arnold2021sharpening}.

\bigskip Let us conclude this paper by some open problems:
\\[4pt]
%%%%%%%%%%%%%%%%%%%%%%%%%%%%%%%%%%%%%%%%%%%%%%%%%%%%%%%%%%%%%%%%%%%%%%%%
$\rhd$ \emph{Unconfined or weakly confined diffusions.} Diffusion equations without any external potential or with an external potential which is insufficient to balance the diffusion are now rather well understood and even the hypocoercive theory in corresponding kinetic equations is essentially under control, see~\cite{BDMMS,BDS-very-weak,BDLS2020}. In~\cite[Section~6]{BDMMS}, improved decay rates are obtained by Fourier estimates when moments of low order are set to zero. A spectral interpretation of these moment conditions, in absence of a bounded stationary solution and the corresponding functional inequalities, written as entropy-entropy production inequalities, is however still missing. Reinterpreting these inequalities in terms of spectral gaps would also be very useful. In kinetic theory, recent results of \cite{armstrong2019variational,brigati2021time} also point in this direction, with an explicit role played by a Poincar\'e-Lions type inequality.
\\[4pt]
%%%%%%%%%%%%%%%%%%%%%%%%%%%%%%%%%%%%%%%%%%%%%%%%%%%%%%%%%%%%%%%%%%%%%%%%
$\rhd$ \emph{Branches of solutions, optimal constants and stability for Sobolev's inequality: a conjecture on the sphere.}
Let us consider the inequality
\be{Ineq:GNS2}
\nrmS{\nabla u}2^2+\frac\lambda{p-2}\,\nrmS u2^2\ge\frac{\mu(\lambda)}{p-2}\,\nrmS up^2\quad\forall\,u\in\mathrm H^1(\S^d,d\mu)
\ee
with optimal constant $\mu(\lambda)$. For an optimal function $u$, the corresponding Euler-Lagrange equation is
\be{EL:Sphere}
-\,(p-2)\,\Delta\,u+\lambda\,u=u^{p-1}
\ee
under the normalization condition $\nrmS up^{p-2}=\mu(\lambda)$. Since the optimal constant $\mu(\lambda)$ is defined as an infimum (in $u$) of an affine function of $\lambda$, $\lambda\mapsto\mu(\lambda)$ is a concave function. The key estimate is the \emph{rigidity} result: $u\equiv\lambda^{1/(p-2)}$ is the unique solution of~\eqref{EL:Sphere} if $\lambda\le d$, and as a consequence we know that $\mu(\lambda)=\lambda$.

The first positive eigenvalue of the Laplace-Beltrami operator $\Delta$ on $\S^d$ is $\lambda_1=d$. Let $\varphi_1$ be an associated eigenfunction. By taking as a test function $u_\varepsilon(z)=1+\varepsilon\,\varphi_1(z)$ in the limit as $\varepsilon\to0_+$, it is easy to see that $\lambda=\lambda_1$ is a \emph{bifurcation} point: there is a branch of non-constant solutions emerging from constant solutions at $\lambda=d$. Hence $d$ can be characterized as the minimum of $\{\lambda>0\,:\,\mu(\lambda)<\lambda\}$, and \emph{symmetry breaking}, in the sense of non-constant optimal functions, occurs if and only if $\lambda>d$. See Fig.~\ref{F2}.

For $p<2^*$, it is not very difficult to establish the existence of optimal functions with \emph{antipodal symmetry}, that is, of solutions which satisfy the condition
\[
u(x)=(-x)\quad\forall\,x\in\S^d\,.
\]
%---------------------------------------------------------------------
\begin{figure}[hb]\begin{center}
\begin{picture}(12,4.5)
\put(2,0){\includegraphics[width=8cm]{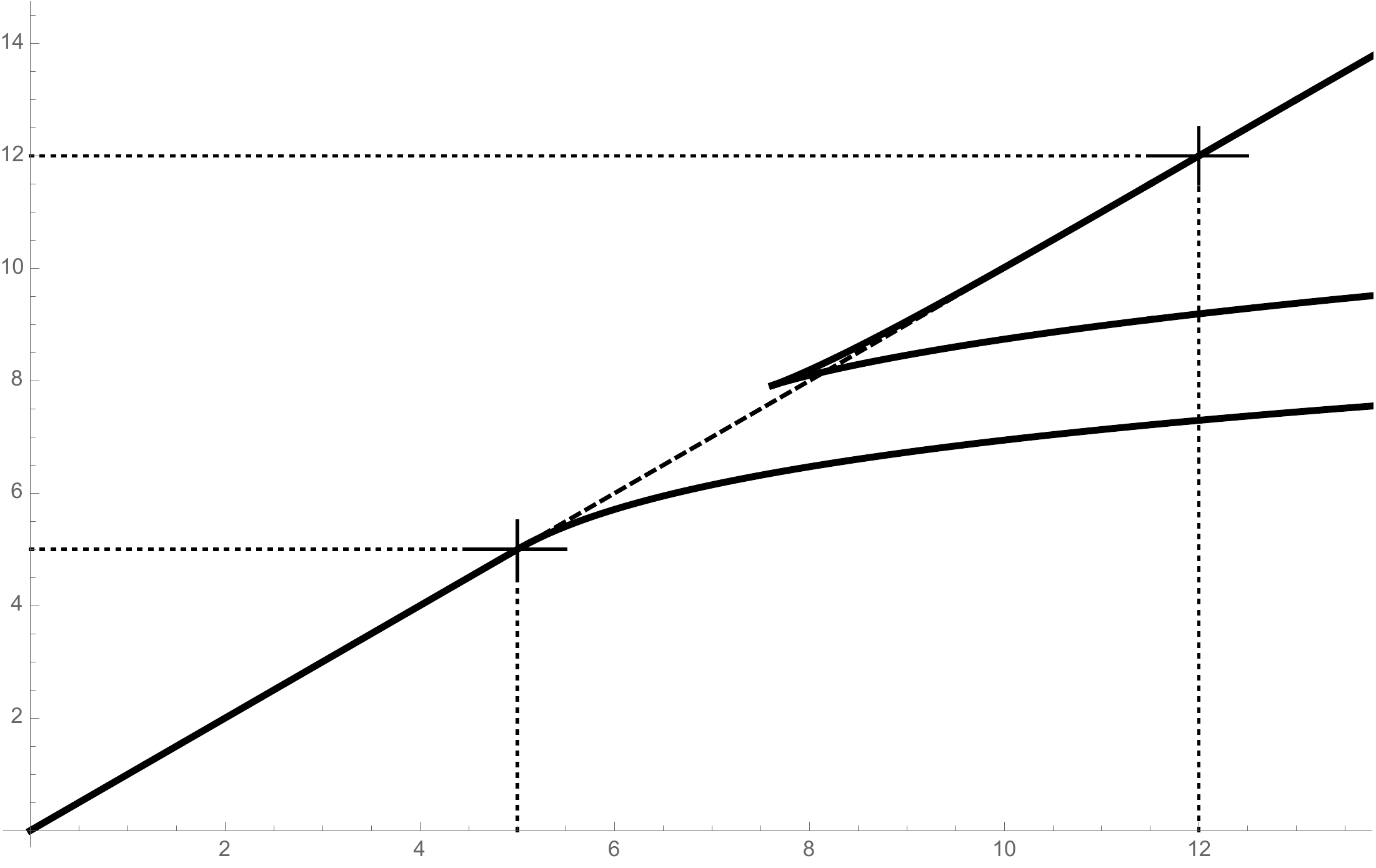}}
\put(9.6,0.5){$\lambda$}
\put(2.4,4.8){$\mu$}
\put(7,2.1){$\mu=\mu(\lambda)$}
\put(8.5,3.5){$\mu=\mu_a(\lambda)$}
\end{picture}
\caption{\label{F2} Case of the sphere in dimension $d=5$: in the subcritical case $p<2^*$, the optimal constant of~\eqref{Ineq:GNS2} is given by $\lambda\mapsto\mu(\lambda)$. The dashed line is the branch of constant solutions $u\equiv1$. The first bifurcation point is at $\lambda_1=d$. From the second bifurcation point, at $\lambda_2=2\,(d+1)$, emerges another branch with a turning point: they intersect at $\lambda=\kappa_p$. The curve $\lambda\mapsto\mu_a(\lambda)$ is obtained as the minimum of this branch and of the line $\mu=\lambda$.}
\end{center}
\end{figure}
%---------------------------------------------------------------------
\clearpage\noindent Let us denote by $\mathrm H^1_a(\S^d,d\mu)$ the subset of $\mathrm H^1(\S^d,d\mu)$ with such a symmetry. We can numerically compute the lowest value for which $\lambda=\mu_a(\lambda)$ among the antipodal solutions $u_\lambda$ of~\eqref{EL:Sphere}, where
\[
\mu_a(\lambda)=\frac{(p-2)\,\nrmS{\nabla u_\lambda}2^2}{\nrmS{u_\lambda}p^2-\nrmS{u_\lambda}2^2}\,.
\]
This can be done using partial symmetries and a convenient reparametrization of the branch emerging from $\lambda_2$. The branch is also shown on Fig.~\ref{F2}. The intersection with the line $\mu=\lambda$ is expected to determine the optimal constant $\kappa_p$ in the inequality
\be{improvedS}
\nrmS{\nabla u}2^2\ge\frac{\kappa_p}{p-2}\(\nrmS up^2-\nrmS u2^2\)\quad\forall\,u\in\mathrm H_a^1(\S^d,d\mu)
\ee
The curve $p\mapsto\kappa_p$ is shown on Fig.~\ref{F3}. The loss of compactness as $p\to2^*$ corresponds to a concentration on two antipodal points. Based on these numerical computations, we expect that $\kappa_{2^*}=2^{1-2/2^*}d=2^{2/d}\,d$. Because of the stereographic projection and after taking into account the normalization of the volume of the sphere ($d\mu$ is a probability measure), this supports the conjecture:
%---------------------------------------------------------------------------
\conj{The optimal constant $\kappa$ in~\eqref{improvedS} for $p=2^*$, on $\S^d$, is $\kappa_{2^*}=2^{2/d}\,d$.}
%---------------------------------------------------------------------------
Such a conjecture is an analogue of the Bianchi-Egnell estimate~\eqref{Bianchi-Egnell} written on~$\S^d$ instead of $\R^d$, where $\mathsf S_d$ is the optimal constant in~\eqref{SobolevRd}. It is compatible with the behaviour of the branches originating from $\lambda_2$, whose lower part converges to an horizontal line as $p\to2^*$. Such a phenomenon is similar to what has been observed in~\cite{0951-7715-27-3-435,FreefemDolbeaultEsteban} for Caffarelli-Kohn-Nirenberg inequalities.
%---------------------------------------------------------------------
\begin{figure}[ht]\begin{center}
\begin{picture}(12,4)
\put(3,0){\includegraphics[width=6cm]{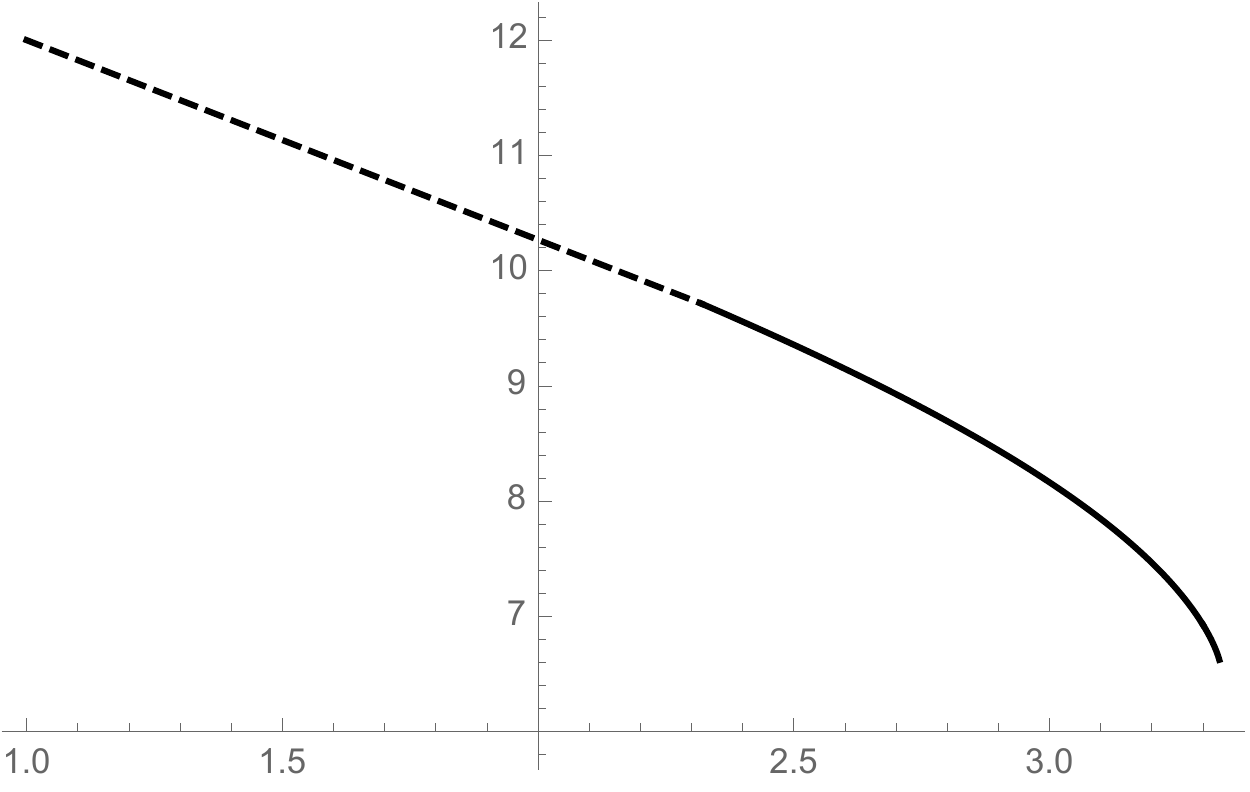}}
\put(9.3,0.2){$p$}
\put(5.8,3.6){$\kappa_p$}
\end{picture}
\caption{\label{F3} A numerical computation of the optimal constant $\kappa_p$ for~\eqref{improvedS} when $d=5$. Only a few points have been computed on the dashed curve. The value of $\kappa_1$ (at $p=1$) is known and equal to $\lambda_2=2\,(d+1)=12$. At $p=2^*$, we get that $\kappa_{10/3}\approx6.59754$.}
\end{center}
\end{figure}
%---------------------------------------------------------------------

%%%%%%%%%%%%%%%%%%%%%%%%%%%%%%%%%%%%%%%%%%%%%%%%%%%%%%%%%%%%%%%%%%%%%%%%
%%%%%%%%%%%%%%%%%%%%%%%%%%%%%%%%%%%%%%%%%%%%%%%%%%%%%%%%%%%%%%%%%%%%%%%%
\subsection*{Acknowledgments}
This work was supported by the Project EFI (ANR-17-CE40-0030) of the French National Research Agency (ANR). \emph{\copyright~2021 by the author. This paper may be reproduced, in its entirety, for non-commercial purposes.}
%%%%%%%%%%%%%%%%%%%%%%%%%%%%%%%%%%%%%%%%%%%%%%%%%%%%%%%%%%%%%%%%%%%%%%%%
%%%%%%%%%%%%%%%%%%%%%%%%%%%%%%%%%%%%%%%%%%%%%%%%%%%%%%%%%%%%%%%%%%%%%%%%
%\nocite*
\bibliographystyle{siam}
\bibliography{Milan2021}
%%%%%%%%%%%%%%%%%%%%%%%%%%%%%%%%%%%%%%%%%%%%%%%%%%%%%%%%%%%%%%%%%%%%%%%%
%%%%%%%%%%%%%%%%%%%%%%%%%%%%%%%%%%%%%%%%%%%%%%%%%%%%%%%%%%%%%%%%%%%%%%%%
%\newpage\tableofcontents
\end{document}